\documentclass[aop,MSNbibl,dvips]{arximspdf}
\usepackage{graphicx}
\doi{10.1214/12-AOP810} %kopijuoti is PTS
\volume{42}
\issue{2}
\pubyear{2014}
\firstpage{689}
\lastpage{724}

\makeatletter
\newtheorem{theorem}{Theorem}
\newproclaim{remark}[theorem]{Remark}
\newtheorem{lemma}[theorem]{Lemma}
\newtheorem{proposition}[theorem]{Proposition}

\newproclaim{cut}{Scale invariance with cut-off}
\newtheorem{corollary}[theorem]{Corollary}
\newproclaim{definition}[theorem]{Definition}
\newproclaim{fact}[theorem]{Fact}
\newproclaim{assumption}[theorem]{Assumption}

\newtheorem{conjecture}[theorem]{Conjecture}
\let\widebar\overline
\makeatother

\begin{document}
\begin{frontmatter}

\title{Levy multiplicative chaos and star scale invariant random measures}
\runtitle{L\'evy multiplicative chaos}

\begin{aug}
\author[A]{\fnms{R\'emi} \snm{Rhodes}\corref{}\thanksref{t3}\ead[label=e1]{rhodes@ceremade.dauphine.fr}},
\author[B]{\fnms{Julien} \snm{Sohier}\ead[label=e2]{j.sohier@tue.nl}}
\and
\author[C]{\fnms{Vincent} \snm{Vargas}\thanksref{t3}\ead[label=e3]{vincent.vargas@ens.fr}}
\runauthor{R. Rhodes, J. Sohier and V. Vargas}
\affiliation{Universit{\'e} Paris-Dauphine, Technische Universiteit Eindhoven  and\\ Ecole Normale sup\'erieure de Paris}
\address[A]{R. Rhodes\\
Ceremade, UMR 7534\\
Universit{\'e} Paris-Dauphine\\
Place du marechal de Lattre de Tassigny\\
75775 Paris Cedex 16\\
France\\
\printead{e1}} %adresu isvedimo komanda gale!
\address[B]{J. Sohier\\
Department of Mathematics\\
\quad and Computer Science\\
Technische Universiteit Eindhoven\\
P.O. Box 513\\
5600 MB Eindhoven\\
The Netherlands\\
\printead{e2}}
\address[C]{V. Vargas\\
Departement de Math\'ematiques et Applications\\
CNRS UMR 8553\\
Ecole Normale sup\'erieure de Paris\\
45 rue d'Ulm\\
F-75 230 Paris Cedex 05\\
France\\
\printead{e3}}
\end{aug}
\thankstext{t3}{Supported in part by the CHAMU ANR project (ANR-11-JCJC).}

% HISTORY:
\received{\smonth{1} \syear{2012}}
\revised{\smonth{10} \syear{2012}}

% ABSTRACT
%
\begin{abstract}
In this article, we consider the continuous analog of the celebrated
Mandelbrot star equation with infinitely divisible weights. Mandelbrot
introduced this equation to characterize the law of multiplicative
cascades. We show existence and uniqueness of measures satisfying the
aforementioned continuous equation. We obtain an explicit
characterization of the structure of these measures, which reflects the
constraints imposed by the continuous setting. In particular, we show
that the continuous equation enjoys some specific properties that do
not appear in the discrete star equation. To that purpose, we define
a~L\'evy multiplicative chaos that generalizes the already existing
constructions.
\end{abstract}

% KEYWORDS
\pdfkeywords{60G57, 28A80, 60H25, 60G15, 60G18, Random measure, star equation, scale invariance,
multiplicative chaos, uniqueness, infinitely divisible processes, multifractal processes}

\end{frontmatter}

%%%%%%%%%%%%%%%%%%%%%%%%%%%%%%%%%%%%%%%%%%%%%%%%%%%%%%%%%%%%%%%%%%%%%%%%%%%%%%%%%
%%%%%%%%%%%%%%%%%%%%% Here the document begins
%%%%%%%%%%%%%%%%%%%%%%%%%%%%%%%%
%%%%%%%%%%%%%%%%%%%%%%%%%%%%%%%%%%%%%%%%%%%%%%%%%%%%%%%%%%%%%%%%%%%%%%%%%%%%%%%%%
%s1 ###
%s1 #&#
\section{Introduction}\label{intro}
%%%%%%%%%%%%%%%%%%%%%%%%%%%%%%%%%%%%%%%%%%%%%%%%%%%%%%%%%%%%%%%%%%%%%%%%%%%%%%%%
Log-normal multiplicative martingales were introduced by Mandelbrot
\cite{mandelfan} in order to build random measures describing energy
dissipation and contribute explaining intermittency effects in
Kolmogorov's theory of fully developed turbulence; see
\cite{cfCastaing,cfSch,cfSto,cfCas,cfFr}. Two years later, Mandelbrot
\cite{mandelbrot} introduced the so-called random multiplicative
cascades as a more easily understandable and more mathematically
tractable alternative. Indeed, this last model was somewhat familiar to
the community of turbulence as some authors \cite{novikov,yaglom} had
already considered, with various degrees of rigor, such cascades in the
restricted framework of the so called conservative case.

Random multiplicative cascades exhibit nonlinear power-law scalings,
rendering intermittency effects in turbulence. Random multiplicative
cascades are \mbox{therefore} the first mathematical discrete approach of
multifractality. Roughly speaking, a (dyadic) multiplicative cascade is
a positive random measure $M$ on the unit interval $[0,1]$ that obeys
the following decomposition rule:
%e1 ###
%e1 #&#
\begin{equation}
\label{starcasc} M(dt)\stackrel{\mathrm{law}}{=}Z^0\mathbf{1}_{[0,1/2]}(t)M^0(2dt)+Z^1
\mathbf{1}_{[1/2,1]}(t)M^1(2dt-1),
\end{equation}
where $M^0,M^1$ are two independent copies of $M$, and $(Z^0,Z^1)$ is a
random vector with prescribed law and positive components of mean $1$
independent from $M^0,M^1$. Such an equation (and its generalizations
to $b$-adic trees for $b\geq2$), the celebrated star equation
introduced by Mandelbrot in \cite{mandelbrotstar}, uniquely determines
the law of the multiplicative cascade. Since the seminal work of
Mandelbrot, the star equation (\ref{starcasc}) has been intensively
studied: of particular interest are the founding paper by Kahane and
Peyriere \cite{cfKahPey} and the work by Durrett and Ligget
\cite{durrett}. The following literature on the topic essentially
builds on these two works. Let us also mention the article
\cite{Comets} which shows that the free energy of a directed polymer
model can be obtained as the limit of the free energy of multiplicative
cascade models, thus establishing a link between the two models.

Despite the fact that multiplicative cascades have been widely used as
reference models in many applications, they possess many drawbacks
related to their discrete scale invariance; mainly they involve a
particular scale ratio, and they do not possess stationary fluctuations
(this comes from the fact that they are constructed on a dyadic tree
structure).

Much effort has been made to develop a continuous parameter theory of
suitable stationary multifractal random measures ever since, stemming
from the theory of multiplicative chaos introduced by Kahane
\cite{cfKah,Bar,cfSch,bacry,cfRoVa,rhovar}. Nevertheless, in comparison
with the discrete case, the state of the art concerning continuous time
models sounds rather empty: laying the foundations like defining a
proper continuous star equation is very recent and its solving only
concerns the lognormal situation \cite{allez}. The main reasons are
technical: first, Gaussian processes are very well understood and,
second, the analysis of Gaussian multiplicative chaos is much
simplified by the use of convexity inequalities for lognormal weights
introduced by Kahane; see Kahane's original paper \cite{cfKah} or
\cite{allez}, Lemma~10, for instance.

In this paper, we are concerned with solving the continuous star
equation:

%
%$\1 #&#
\begin{stareq*}
A stationary random measure $M$ on $\mathbb{R}^d$ is said to be $\star$-scale
invariant if for all $0<\epsilon\leq1$, $M$ obeys the cascading rule
%e2 ###
%e2 #&#
\begin{equation}
\label{star} \bigl(M(A) \bigr)_{A\in\mathcal{B}(\mathbb{R}^d)}\stackrel
{\mathrm{law}} {=} \biggl(\int
_Ae^{\omega_{\epsilon}(r)}M^{\epsilon}(dr)
\biggr)_{A\in\mathcal{B}(\mathbb{R}^d)},
\end{equation}
%
%,\qquad\big(M(A)\big)_{A\in\mathcal{B}(\R^d)}=\big(\int_Ae^{\omega_{
where $\omega_{\epsilon}$ is a stochastically continuous stationary
process and $M^{\epsilon}$ is a random measure independent from
$\omega_{\epsilon}$ satisfying the relation
\[
\bigl(M^{\epsilon}(\epsilon A) \bigr)_{A\in\mathcal{B}(\mathbb
{R}^d)}\stackrel{\mathrm{law}}{=}
\epsilon^d \bigl(M(A) \bigr)_{A\in\mathcal
{B}(\mathbb{R}^d)}.
\]
\end{stareq*}
Intuitively, this relation means that when you zoom in the measure $M$,
you should
observe the same behavior up to an independent factor. Notice that this
definition is stated in great generality since no constraint on the law
of $\omega_\epsilon$ is imposed. In the context of discrete
multiplicative cascades, given any law for $\omega_\epsilon$ (up to
some integrability conditions), this equation can be solved. However,
the continuous case imposes the following constraint on $\omega
_{\epsilon}$:

%le1 #&#
\begin{lemma}\label{lemme_id}
We consider a nontrivial $\star$-scale invariant measure $M$ on
$\mathbb{R}^d$. We suppose that for some $x$ (and hence all $x$) the
family $\epsilon\rightarrow\omega_{\epsilon}(x)$ is continuous in
distribution and
\[
\mathbb E \bigl[ \bigl( M[0,1]^d\bigr)^{\gamma}\bigr] <
\infty,\qquad\mathbb E \bigl[e^{(1+\gamma)\omega_{\epsilon}(x)}\bigr]<
\infty\qquad\forall\epsilon\leq1
\]
for some $\gamma>0$. Then, for all $\epsilon$, the process $\omega
_{\epsilon}$ is infinitely divisible.
\end{lemma}
Hence, with minimal assumptions on $\omega_{\epsilon}$ and the solution
$M$, the process $\omega_{\epsilon}$ is infinitely divisible. In view
of the above lemma, we can suppose that the process $\omega_{\epsilon}$
is infinitely divisible: we will make this assumption in the sequel. As
suggested by the Gaussian case \cite{allez}, this naturally leads to
the issue of constructing random measures formally defined by
\[
M(dx)=e^{L_x} \,dx,
\]
where the process $L$ is infinitely divisible with logarithmic
correlations. We carry out this construction in Section~\ref{genchaos},
which generalizes already existing such attempts
\cite{bacry,Bar,fan,rhovar}. We call such measures L\'evy
multiplicative chaos. This construction enables us not only to give
nontrivial solutions to (\ref{star}) (in Section~\ref{secstar}) but
also to characterize all the solutions to (\ref{star}) (up to a few
additional technical assumptions). These solutions share the property
of a specific structure for the law of the process~$\omega_{\epsilon}$.
This structure reflects the fact that the continuous star equation is
far more restrictive than the discrete one (similarly, L\'evy processes
are in some sense more restrictive than discrete simple random walks
which can be considered with any law for the increments).

%s1.1 ###
%s1.1 #&#
\subsection{Notation}
%%%%%%%%%%%%%%%%%%%%%%%%%%%%%%%%%%%
We will use the following notation throughout the paper.
$\mathcal{B}(E)$ stands for the Borel $\sigma$-field of a topological
space $E$. A random measure $M$ is a random variable taking values into
the set of positive Radon measures defined on $\mathcal
{B}(\mathbb{R}^d)$. We will say that $M$ possesses a moment of order
$p>0$ if $ \mathbb E [M(K)^p]<+\infty$ for every compact set $K$. A
random measure $M$ is said to be stationary if for all
$y\in\mathbb{R}^d$ the random measures $M(\cdot)$ and $M(y+\cdot) $
have the same law. A stochastic process $(X_t)_{t\in\mathbb{R}^d}$ is
said to be stochastically continuous if, for each $t\in\mathbb{R}^d$,
$X_{t+h}$ converges toward $X_t$ in probability when $h$ goes to $0$.
We will also use the shortcut ID in place of infinitely divisible. We
remind the reader that every stochastically continuous random process
admits a measurable version; see \cite{borkar}, Chapter~6. We will only
deal with measurable versions of stochastically continuous process in
this paper.

%
% \begin{lemma}
% Let $(F(x))_{x \in\R^d}$ and $(G(x))_{x \in\R^d}$ be two stationary
%and stochastically continuous processes. We consider a Random measure $
%positive mass to all open sets. We suppose that there exists $
% and $E[ \eta(K)^{1+\gamma}] < \infty$ for all compact set $K$. If the
%following equality on measures holds:
% \[
% F(x) \eta(dx) \stackrel{law}{=} G(x) \eta(dx)
% \]
% then the two processes $F$ and $G$ have same law.
% \end{lemma}

%%%%%%%%%%%%%%%%%%%%%%%%%%%%%%%%%%%%%%%%%%%%%%%%%%%%%%%%%%%%%%%%%%%%%%%%%%%%%%%%%%%%%%%%%%%%%%%%%%%%%%%%%%%%%%%%%%%%%%%%%%%%%%%%%
%s2 ###
%s2 #&#
\section{Generalized L\'evy chaos}\label{genchaos}
%%%%%%%%%%%%%%%%%%%%%%%%%%%%%%%%%%%%%%%%%%%%%%%%%%%%%%%%%%%%%%%%%%%%%%%%%%%%%%%%%%%%%%%%%%%%%%%%%%%%%%%%%%%%%%%%%%%%%%%%%%%%%%%%%
This section is devoted to the construction of measures that can
formally be written as
\[
M(dx)=e^{L_x} \,dx,
\]
where $L$ is a stationary ID process with a logarithmic spatial
dependency. As in the Gaussian case, such a singularity of the spatial
structure imposes to construct these measures through a limiting
procedure where the singularity has been ``cut off.'' Hence we will
understand these measures as a limit
\[
M(dx)=\lim_{\epsilon\to0}e^{X^\epsilon_x} \,dx,
\]
where $X^\epsilon$ is a stationary ID process that converges in some
sense toward $L$. The process $X^\epsilon$ will basically depend on
two parameters: a generator (any stationary ID process) and a rate
function. We detail below the construction.

%s2.1 ###
%s2.1 #&#
\subsection{Generator and rate function}
%%%%%%%%%%%%%%%
Let $(X_t)_{t\in\mathbb{R}^d}$ be a stochastically continuous
stationary ID random process. It follows from \cite{maru} that $X$
admits a version given by
%e3 ###
%e3 #&#
\begin{eqnarray}\label{version}
X_t &=& b+\int_{\mathbb{R}^d}\cos(t\cdot
\lambda) \widebar{W}(d\lambda)+\int_{\mathbb{R}^d}\sin(t\cdot\lambda)
\widebar{W}'(d\lambda)
\nonumber\\[-8pt]\\[-8pt]
&&{} +\int_Sf \bigl(T_t(s)\bigr) \bigl[\widebar{N}(ds)-\bigl(1\vee\bigl|f
\bigl(T_t(s)\bigr)\bigr|\bigr)^{-1}\theta(ds)\bigr],\nonumber
\end{eqnarray}
where:
\begin{itemize}
\item $b\in\mathbb{R}$;

\item $\widebar{W},\widebar{W}',\widebar{N}$ are independent;

\item $\widebar{W},\widebar{W}'$ are identically distributed
    centered Gaussian random measures on $\mathbb{R}^d$ with
    covariance kernel given by $ \mathbb E
    [\widebar{W}(A)\widebar{W}(B)]=R(A\cap B)$ for some symmetric
    positive finite measure $R$ on
    $(\mathbb{R}^d,\mathcal{B}(\mathbb{R}^d))$;

\item$\widebar{N}$ is a Poisson random measure on a Borel space $S$
    with a $\sigma$-finite intensity measure $\theta$;

\item$f\dvtx S\to\mathbb{R}$ is a measurable deterministic function
    such that
\[
\int_S\bigl(\bigl|f(s)\bigr|^2\wedge1\bigr) \theta(ds)<+
\infty;
\]

\item$(T_x)_x$ is a measure preserving flow on $(S,\theta)$.
\end{itemize}
In what follows, we will say that a stochastically continuous ID
process is associated with
$(S,\widebar{W},\widebar{W}',\widebar{N},\theta,R,f,(T_x)_x)$ if it is
given by (\ref{version}) where all the involved items are defined as
described above.

We define the Laplace exponents $ \psi$ of $X$ for $p\geq1$ by
\[
\mathbb E \bigl[e^{q_1X_{t_1}+\cdots+q_pX_{t_p}}\bigr]=e^{\psi_{t_1,\ldots,t_p}(q_1,\ldots,q_p)}
\]
for all $(t_1,\ldots,t_p)\in(\mathbb{R}^d)^p$ and $q_1,\ldots,q_p\in
\mathbb{R}$ such that the above expectation makes sense. For the sake
of clarity, $\psi_0$ (i.e., the Laplace exponents of $X_0$, or
equivalently of $X_t$ for any $t\in\mathbb{R}^d$) will be denoted by
$\psi$.\vadjust{\goodbreak}

We assume that $X$ possesses a second order exponential moment, and we
consider the following generalized covariance function:
%e4 ###
%e4 #&#
\begin{eqnarray}
\label{defF} F(x)&=& \psi_{0,x}(1,1)-2\psi(1),\qquad x\in
\mathbb{R}^d.
\end{eqnarray}

%as2 #&#
\begin{assumption}\label{vitg}
Let $g$ be a nonnegative function in $L^1_{\mathrm{loc}}(\mathbb{R}_+,dy)$ such that
%e5 ###
%e5 #&#
\begin{eqnarray}
&& \forall x \in\mathbb{R}^d\setminus\{0\}\quad\mbox{and}\quad a\geq1
\nonumber\\[-8pt]\\[-8pt]
&& \qquad\int_a^{+\infty}\frac{|F(g(u)x)|}{u} \,du\leq
\widebar{F}\ln_+\frac{1}{a|x|}+h(a,x),\nonumber
\end{eqnarray}
where $h$ is some bounded continuous function on $\mathbb{R}_+\times
\mathbb{R}^d$ and $\widebar{F}$ is some positive constant. The function
$g$ will be called rate function.
\end{assumption}

%s2.2 ###
%s2.2 #&#
\subsection{Limiting procedure}
%%%%%%%%%%%%%%%%%%%%%%%%%
For any $\epsilon\in\,]0,1[$, we define a new stochastically continuous
ID random process:
%e7 ###
%e6 ###
%e6 #&#
%e7 #&#
\begin{eqnarray}
 X^\epsilon_t&=&b\ln\frac{1}{\epsilon}+\int
_1^{1/\epsilon}\!\!\!\int_{\mathbb{R}^d}\cos
\bigl(tg(y)\cdot\lambda\bigr) W(d\lambda,dy)
\nonumber\\[-8pt]\label{version2} \\[-8pt]
&&{} +\int_1^{1/\epsilon}\!\!\!
\int_{\mathbb{R}^d}\sin\bigl(tg(y)\cdot\lambda\bigr)
W'(d\lambda,dy)\nonumber
\\
&&{}+\int_1^{1/\epsilon}\!\!\!\int_Sf
\bigl(T_{tg(y)}(s)\bigr) \biggl[N(ds,dy)-\bigl(1\vee\bigl|f
\bigl(T_{tg(y)}(s)\bigr)\bigr|\bigr)^{-1}\theta(ds)
\frac{dy}{y}\biggr],\hspace*{-24pt}
\end{eqnarray}
where:
\begin{itemize}
\item $W,W',N$ are independent;

\item $W,W'$ are identically distributed centered Gaussian random measures
on $\mathbb{R}^d\times\mathbb{R}^*_+$ with covariance kernel
$R(d\lambda)\frac{dy}{y}$, that is, $ \mathbb E [W(A)W(B)]=\int_{A\cap
B}R(d\lambda)\frac{dy}{y}$ for any Borel sets $A,B\subset
\mathbb{R}^d\times\mathbb{R}^*_+$;

\item $N$ is a Poisson random measure on the Borel space $S\times\mathbb
{R}^*_+$ with intensity measure $\theta(ds)\otimes\frac{dy}{y}$;

\item $f\dvtx S\to\mathbb{R}$ and $(T_x)_x$ are the same as above.
\end{itemize}

Clearly, $X^\epsilon$ is a stationary ID process. From \cite{maru},
Theorem~5, it is stochastically continuous. In what follows, we will
say that a family $(X^\epsilon)_\epsilon$ of stationary stochastically
continuous ID processes is an approximating family associated with
$(S,W,W',N,\theta,R,f,(T_x)_x)$ if it is given by (\ref{version2})
where all the involved items are defined as described above. Notice
that the whole law of the processes $(X^\epsilon)_{\epsilon\in]0,1]}$
can be recovered from the law of the process $X$ introduced in the
previous subsection and the rate function $g$. For this reason, the ID
process $X$ will be called the \textit{generator} of the approximating
sequence $(X^\epsilon)_\epsilon$ and $g$ the \textit{rate function}.

We have
%e8 ###
%e8 #&#
\begin{equation}
\label{defpsi} \forall q\geq0,\ \forall x\in\mathbb{R}^d\qquad\mathbb E
\bigl[e^{qX^\epsilon_x}\bigr]= \mathbb E \bigl[e^{qX^\epsilon_0}
\bigr]=e^{\ln(1/\epsilon)\psi(q)}.
\end{equation}
We stress that, in great generality, $\psi$ takes values into $\mathbb
{R}_+\cup
\{+\infty\}$, but it is finite at least for $q\in[0,2]$.

For $\epsilon>0$, we define a random measure
%e9 ###
%e9 #&#
\begin{equation}
\label{mart} \forall A\in\mathcal{B}\bigl(\mathbb{R}^d\bigr)\qquad
\widetilde{M}{}^\epsilon(A)=\int_Ae^{X^\epsilon_x-\psi(1)\ln(1/\epsilon)}
\,dx.
\end{equation}
Clearly, for each fixed $A$ with finite Lebesgue measure, the family
$(\widetilde{M}{}^\epsilon(A))_{\epsilon\in]0,1[}$ is a positive
martingale. Thus it converges almost surely. We deduce that the family
$(\widetilde{M}{}^\epsilon)_\epsilon$ almost surely weakly converges
toward a limiting random measure $M$ on $\mathcal{B}(\mathbb{R}^d)$.
This measure will be called L\'evy multiplicative chaos associated with
$(S,W,W',N,\theta,R,f,(T_x)_x)$.

%s2.3 ###
%s2.3 #&#
\subsection{Main properties}
%%%%%%%%%%%%%%%%%%%%
By stationarity and the $0-1$ law, we deduce (as~in
\mbox{\cite{cfKah,rhovar2}}).
%
%pr3 #&#
\begin{proposition}\label{Ndeg}
Either of the following events occurs with probability one:
\[
\{M\equiv0\}\quad\mbox{or}\quad\bigl\{\forall B\mbox{ nonempty ball }M(B)>0
\bigr\}.
\]
In the second situation, we will say that the measure $M$ is
nondegenerate.
\end{proposition}

The nondegeneracy is expectedly related to the Laplace exponents of the
generator:
%
%th4 #&#
\begin{theorem}\label{a_mont}
Under Assumption~\ref{vitg}, the measure $M$ is nondegenerate as soon
as $\psi'(1) -\psi(1)< d$.
\end{theorem}

%co5 #&#
\begin{corollary}
Under Assumption~\ref{vitg} and provided that $\psi'(1) -\psi(1)< d$,
the measure $M$ almost surely does not possess any atom.
\end{corollary}

In some particular situations, it can be proved that the condition
$\psi'(1) - \psi(1)< d$ is optimal; see \cite{cfKah,Bar,bacry}, for
instance. But the situation presented here is far more intricate and it
is not optimal in great generality since we only require the
correlation structure to be sub-logarithmic (Assumption~\ref{vitg}). To
illustrate the situation, let us focus on the second order moment. It
is well known that, in the particular situations presented in
\cite{cfKah,Bar,bacry}, the measure $M$ admits a second order moment if
and only if $\psi(2)<d$. In our case, the situation is not that clear.
For instance, choose $\theta$ equal to the Lebesgue measure on
$S=\mathbb{R}^d$, $\theta$ any L\'evy measure on $\mathbb{R}^d$ and
$R=0$. The flow $(T_t)_t $ is the usual group of translations. Take any
positive bounded function $f$ with compact support over $\mathbb{R}^d$
and $g(y)=y^q$ (for $q\geq1$). Notice that the associated function $F$
reduces to $0$ for all $x$ such that the supports of $f$ and $T_xf$ are
disjoint, say for $|x|\geq R$. Then for $a>0$ and
$x\in\mathbb{R}^d\setminus\{0\} $, we have (where $e_x=x/|x|$)
%e10 #&#
\begin{eqnarray}\label{rem5}
\qquad\quad \int_1^{+\infty}\frac{F(u^qx)}{u} \,du &=&\int
_{|x|^{1/q}}^{+\infty
}\frac{F(u^qe_x)}{u} \,du
\nonumber\\[-8pt]\\[-8pt]
&\simeq&\frac{F(0)}{q}\ln_+\frac{1}{|x|}=\frac{\psi(2)-2\psi
(1)}{q}\ln_+
\frac{1}{|x|} \qquad\mbox{as }x\to0.\nonumber
\end{eqnarray}
Hence it can be proved that $M$ admits a second order moment if and
only if $\frac{\psi(2)-2\psi(1)}{q}<d$, which is quite a different
condition from \cite{cfKah,Bar,bacry}.

Hence, it appears that the condition $\psi'(1) -\psi(1)< d$ should be
optimal when the rate function $g$ is ``not far'' from the function
$g(y)=y$. In that spirit, we claim:
%
%th6 #&#
\begin{theorem}\label{cnsq}
If the measure $M$ admits a moment of order $1+\delta$ for some
$\delta>0$, and if the rate function $g$ satisfies $g(y)\leq y$ for
$y\geq1$, then
\[
\psi(1+\delta)-(1+\delta)\psi(1)\leq d\delta.
\]
In particular $\psi'(1) -\psi(1)< d$.
\end{theorem}

%%%%%%%%%%%%%%%%%%%%%%%%%%%%%%%%%%%%%%%%%%%%%%%%%%%%%%%%
%s2.4 ###
%s2.4 #&#
\subsection{On possible generalizations}
%%%%%%%%%%%%%%%%%%%%%%%%%%%%%%%%%%%%%%%%%%%%%%%%%%%%%%%%
In the spirit of \cite{cfKah}, it is possible to make the
multiplicative chaos act on other measures than the Lebesgue measure.
More precisely, choose a Radon measure $\kappa$ on $\mathbb{R}^d$ and,
for $\epsilon>0$, define the random measure
%e10 ###
%e11 #&#
\begin{equation}
\label{martkappa} \forall A\in\mathcal{B}\bigl(\mathbb{R}^d\bigr)\qquad
\widetilde{M}{}^\epsilon(A)=\int_Ae^{X^\epsilon_x-\psi(1)\ln(1/\epsilon)}
\kappa(dx).
\end{equation}
For each fixed $A$ with finite $\kappa$-measure, the family
$(\widetilde{M}{}^\epsilon(A))_{\epsilon\in]0,1[}$ is a positive
martingale once again and therefore converges almost surely. We deduce
that the family $(\widetilde{M}{}^\epsilon)_\epsilon$ almost surely
weakly converges toward a limiting random measure $M$ on
$\mathcal{B}(\mathbb{R}^d)$. This measure will be called L\'evy
multiplicative chaos associated with $(S,W,W',N,\theta,R,f,(T_x)_x)$
and integrating measure $\kappa$.

A $0-1$ law argument shows:
%
%pr7 #&#
\begin{proposition}\label{Ndeg}
Either of the following events occurs with probability one:
\[
\{M\equiv0\}\quad\mbox{or}\quad\bigl\{\forall B\mbox{ nonempty ball
with }
\kappa(B)>0, M(B)>0\bigr\}.
\]
In the second situation, we will say that the measure $M$ is nondegenerate.
\end{proposition}

In this case, nondegeneracy results in an intricate way from the
structure of the measure $\kappa$ as well as the Laplace exponents of
the generator. It seems difficult to state quite generally a result.
Nevertheless we can focus on the situation when the structure of
$\kappa$ is related to the Euclidean metric in the following way:
%
%de8 #&#
\begin{definition}
We introduce the set $R_{\alpha}$ of Radon measures $\nu$ on $\mathbb
{R}^d$ such
that: for any nonempty ball $B$ of $\mathbb{R}^{d}$, for any
$\varepsilon> 0$, there
exist $\delta> 0, D > 0$, and a~compact set $K_{\varepsilon} \subset
B$ with
$\nu(B \setminus K_{\varepsilon}) < \varepsilon$ such that the
measure $\nu_{\varepsilon} =
\mathbf{1}_{K_{\varepsilon}}(x) \nu(dx)$ satisfies, for every
open set $U\subset B$,
%e11 ###
%e30 ###
%e12 #&#
\begin{equation}
\label{inegnu1} \nu_{\varepsilon}(U) \leq D \times \operatorname{diam}(U)^{\alpha+
\delta}.
\end{equation}
\end{definition}
In a rough sense, the class $R_\alpha$ consists of Radon measures that
are locally ``\mbox{$\alpha$-}H\"older.'' For instance, the Lebesgue measure
on $\mathbb{R}^d$ is in the class $R_d$.

%th9 #&#
\begin{theorem}\label{AMontkappa}
Under Assumption~\ref{vitg}, the measure $M$ is nondegenerate as soon
as $\psi'(1) -\psi(1)< \alpha$.
\end{theorem}

%co10 #&#
\begin{corollary}
Under Assumption~\ref{vitg} and provided that $\psi'(1) -\psi(1)<
\alpha$, the measure $M$ almost surely does not possess any atom.
\end{corollary}

%%%%%%%%%%%%%%%%%%%%%%%%%%%%%%%%%%%%%%%%%%%%%%%%%%%%%%%%%%%%%%%%%%%%%%%%%%%%%%%%%%%%%%%%%%%%%%%%%%%%%%%%%%%%%%%%%%%%%%%%%%%%%%%%%
%s3 ###
%s3 #&#
\section{Star scale invariant random measures}\label{secstar}
%%%%%%%%%%%%%%%%%%%%%%%%%%%%%%%%%%%%%%%%%%%%%%%%%%%%%%%%%%%%%%%%%%%%%%%%%%%%%%%%%%%%%%%%%%%%%%%%%%%%%%%%%%%%%%%%%%%%%%%%%%%%%%%%%
In this section, we explain the connection between $\star$-scale
invariant random measures and L\'evy multiplicative chaos. On one hand,
we show that every L\'evy multiplicative chaos defines a \mbox{$\star$-}scale
invariant random measure provided that the rate function is defined by
$g(y)=y$ for all $y\geq1$. Then we show that all $\star$-scale
invariant random measures with a moment of order strictly greater than
$1$ are L\'evy multiplicative chaos, up to a few additional assumptions.

%s3.1 ###
%s3.1 #&#
\subsection{Construction}
%%%%%%%%%%%%%%%%%%%
We consider $X^\epsilon$, $\widetilde{M}{}^\epsilon$ and $M$ as
constructed in Section~\ref{genchaos} with generator $X$ and rate
function $g$ given by $g(y)=y$ for all $y\geq1$. Hence the process
$X^\epsilon$ is given by
%e13 ###
%e12 ###
%e13 #&#
%e14 #&#
\begin{eqnarray}
\qquad X^\epsilon_t&=&b\ln\frac{1}{\epsilon}+\int
_1^{1/\epsilon}\!\!\!\int_{\mathbb{R}^d}\cos(ty\cdot
\lambda) W(d\lambda,dy)
\nonumber\\[-8pt]\label{defX}  \\[-8pt]
&&{} +\int_1^{1/\epsilon}\!\!\!\int_{\mathbb{R}^d}\sin(ty\cdot\lambda) W'(d\lambda,dy)\nonumber
\\
&&{}+\int_1^{1/\epsilon}\!\!\!\int_Sf
\bigl(T_{ty}(s)\bigr) \biggl[N(ds,dy)-\bigl(1\vee\bigl|f
\bigl(T_{ty}(s)\bigr)\bigr|\bigr)^{-1}\theta(ds)\frac{dy}{y}
\biggr].
\end{eqnarray}

Let us state a simple criterion to check Assumption~\ref{vitg}:
%
%pr11 #&#
\begin{proposition}\label{Fstar}
Assumption~\ref{vitg} is satisfied if and only if
%e14 ###
%e15 #&#
\begin{equation}
\label{sphere} \sup_{|e|=1} \int_1^{+\infty}
\frac{|F(ue)|}{u} \,du<+\infty.
\end{equation}
\end{proposition}

%th12 #&#
\begin{theorem}\label{propscaleinv}
Assume that Assumption~\ref{vitg} [or equivalently (\ref{sphere})]
holds and that $\psi'(1)-\psi(1)<d$. Then $M$ is nontrivial and
$\star$-scale invariant.
\end{theorem}
Hence, the $\star$-scale invariance property only depends on the choice
of the rate function. This shows in a way that there are as many
$\star$-scale invariant random measures as stochastically continuous ID
processes [up to the condition $\psi'(1)-\psi(1)<d$].

The existence of a second order moment is ruled by the following
condition, which seems to be more conventional than the counter-example
described in (\ref{rem5}):
%
%pr13 #&#
\begin{proposition}\label{mom2}
The measure $M$ admits a second order moment if and only if $F(0)<d$.
\end{proposition}

A straightforward adaptation of our proofs shows the following:
%
%pr14 #&#
\begin{proposition}
A $\star$-scale invariant random measure $M$ is multifractal in the
sense that
\[
\lim_{t\downarrow0}\frac{\ln\mathbb E [M([0,t])^q]}{\ln t}=q-\psi
(q)+q\psi(1),
\]
where $\psi$ is the Laplace exponent of its generator.
\end{proposition}

%s3.2 ###
%s3.2 #&#
\subsection{Uniqueness}
%%%%%%%%%%%%%%%%%%%
Conversely, we now want to describe as exhaustively as possible the set
of all $\star$-scale invariant random measures. For that purpose, we
introduce a few additional assumptions:
%
%as15 #&#
\begin{assumption}\label{good}
We will say that a stationary random measure $M$ is a good $\star
$-scale invariant random measure if $M$ is $\star$-scale invariant
and satisfies:
\begin{longlist}[(2)]
\item[(1)] the process $\omega_\epsilon$ admits exponential
moments of order $2$, that is,\break
\mbox{$ \mathbb E [e^{2\omega_\epsilon(0)}]<\infty$};

\item[(2)] for $\epsilon<1$, the generalized covariance kernel
associated with the ID process~$\omega_{\epsilon}$
\[
\forall x\in\mathbb{R}^d\qquad F_\epsilon(x) = \log\bigl(
\mathbb E \bigl[e^{\omega_{\epsilon}(x) + \omega_{\epsilon}(0)}\bigr
]\bigr)
\]
satisfies
%e15 ###
%e16 #&#
\begin{equation}
\forall x\neq0\qquad \bigl|F_{\epsilon}(x)\bigr|\leq C_{\epsilon}\int
_{|x|}^{+\infty}\theta(u) \,du \label{modulus}
\end{equation}
%
% \begin{eqnarray}
%|F_{\epsilon}(x)|&\to0 & \text{ as }\quad|x| \to+\infty,\label{lim}\\
% |F_{\epsilon}(x)-F_{\epsilon}(x')|&\leq C_{\epsilon}\theta\big(
for some positive constant $C_{\epsilon}$ and some decreasing
function $\theta\dvtx ]0,+\infty[\,\to\mathbb{R}_+$ such that
%e16 ###
%e17 #&#
\begin{equation}
\label{cvint} \int_1^{+\infty}\theta(u)\ln(u) \,du<+
\infty;
\end{equation}

\item[(3)] there is $\epsilon_0\in\,]0,1]$ such that, for each
    $p\geq1$, $q_1,\ldots,q_p\in\mathbb{R}$ and $t_1,\ldots,t_p\in
    \mathbb{R}^p$, the mapping
\[
(\epsilon,t_1,\ldots,t_p) \mapsto\mathbb E
\bigl[e^{iq_1\omega_\epsilon
(t_1)+i\omega_\epsilon(t_p)}\bigr]
\]
admits a partial derivative w.r.t. $\epsilon$ at $\epsilon=\epsilon
_0$ with a partial derivative continuous w.r.t. $(t_1,\ldots,t_p)$.
\end{longlist}
\end{assumption}
It turns out that the condition on the exponential moments of order $2$
of $\omega_\epsilon$ is also necessary as soon as the measure $M$
possesses a moment of order $2$. Point~2 is a decorrelation property at
infinity whereas point 3 is a regularity property.
%Furthermore, condition \eqref{lipschitz} together with \eqref{lim}
%entails
In what follows, we denote by $\psi_\epsilon$ the Laplace exponent of
$\omega_\epsilon$,
\[
\psi_\epsilon(q)=\ln\mathbb E \bigl[e^{q\omega_\epsilon(0)}\bigr]
\]
for all $q\in\mathbb{R}$ such that the above quantity is finite. Notice
that, as soon as the measure $M$ possesses a moment of order $1$, the
condition $\psi_\epsilon(1)=0$ is a necessary condition for the
solution of (\ref{star}) to be nontrivial.

The main result of this paper is the following:
%
%th16 #&#
\begin{theorem}\label{thuniq}
Consider a good $\star$-scale invariant measure $M$. Assume that $M$
admits a finite moment of order $1+\delta$ for some $\delta>0$\break  (i.e., $
\mathbb E [M(B)^{1+\delta}]<\infty$ for some open ball $B$). Then there
exists a random variable $Y\in L^{1+\delta}$ and a L\'evy
multiplicative chaos $Q$ (independent from $Y$ and nondegenerate) with
associated rate function $g(y)=y$ such that
\[
M(dx)\stackrel{\mathrm{law}}{=}YQ(dx).
\]
\end{theorem}

We conjecture that the same theorem holds if $M$ is a $\star$-scale
invariant measure with a finite moment of order $1+\delta$ for some
$\delta>0$. Therefore, we think Assumption~\ref{good} is just a
technical assumption (which we cannot avoid at present) and that our
theorem characterizes all $\star$-scale invariant measure with a finite
moment of order $1+\delta$ for some $\delta>0$. The general case of
$\star$-scale invariant measures with no finite moment assumption is
currently under investigation and requires the introduction of a
different set of measures (work in progress).

%re17 #&#
\begin{remark}
When $M$ is a good $\star$-scale invariant random measure, the law of~$M$ is entirely characterized by the law of the process
$\omega_\epsilon$ in (\ref{star}) for some \mbox{$\epsilon\in\,]0,1[$}.
Furthermore, the law of the finite-dimensional distributions of the
generator $X$ can be recovered from\vadjust{\goodbreak} those of $\omega_\epsilon$ by the
following procedure: define the L\'evy exponents $\eta^\epsilon$,
$\eta$ of $\omega_\epsilon$ and $X$, that is,
\begin{eqnarray*}
\mathbb E \bigl[e^{iq_1\omega_\epsilon(t_1)+\cdots+iq_p\omega_\epsilon
(t_p)}\bigr]&=&e^{\eta^\epsilon(q_1,\ldots,q_p,t_1,\ldots,t_p)},
\\
\mathbb E \bigl[e^{iq_1X_{t_1}+\cdots+iq_pX_{t_p}}\bigr]&=&e^{\eta(q_1,\ldots,q_p,t_1,\ldots,t_p)}.
\end{eqnarray*}
Then we have
\[
\partial_\epsilon\eta^\epsilon(q_1,\ldots,q_p,t_1,\ldots,t_p)=-\frac{1}{\epsilon^2}
\eta\biggl(q_1,\ldots,q_p,\frac{t_1}{\epsilon
},\ldots,\frac{t_p}{\epsilon}\biggr).
\]
\end{remark}
%
%%%%%%%%%%%%%%%%%%%%%%%%%%%%%%%%%%%%%%%%%%%%
%%%%%%%%%%%%%%%%%%%%%%%%%%%%%%%%%%%%%%%%%%%%%
%In our definition (\ref{star}) of $\star$-scale invariant measures, we
%require the process $(\omega_{\epsilon}(x))_{x \in\R^d}$ to be ID for
%all $\epsilon< 1$. In fact, this is no restriction as we now show. To
%see this, assume that we are given a stochastically continuous process
%$(\omega_{\epsilon}(x))_{x \in\R^d,\epsilon\in]0,1]}$ satisfying
%(\omega_{\epsilon\epsilon'}(x))_{x \in\R^d}\stackrel{law}{=} (
%where $\omega_{\epsilon}$ and $\tilde{\omega}_{\epsilon'}$ are
%independent copies of $\omega_{\epsilon}$ and $\omega_{\epsilon'}$. We
%fix $\epsilon$ and consider $\epsilon_n={\epsilon}^{\frac{1}{n}}$. Of
%course $\epsilon_n^n=\epsilon$. By iterating the cascade rule (
%(\omega_{\epsilon}(x))_{x \in\R^d}\stackrel{law}{=} (\sum_{k=0}^{n-1}
%where the $\omega_{\epsilon_n}^{(k)}$ are independent processes of law
%$\omega_{\epsilon_n}$. Fix $x,y \in\R^d$. We therefore have for all $
%)+\lambda\omega_{\epsilon_n}^{(k)}(\frac{y}{\epsilon_n^k})
%The stochastic continuity of the process $\omega$ with respect to $
%By a classical theorem on independent triangular arrays (see chapter
%XVII in \cite{Feller}), this shows that the couple $( \omega_{
%show that, for all $(x_1, \cdots, x_n)$, the vector $(\omega_{
%
%%%%%%%%%%%%%%%%%%%%%%%%%%%%%%%%%%%%%%%%%
%%%%%%%%%%%%%%%%%%%%%%%%%%%%%%%%%%%%%%%%%
%s4 ###
%s4 #&#
\section{Examples}
%%%%%%%%%%%%%%%%%%%%%%%%%%%%%%%%%%%%%%%%%
%%%%%%%%%%%%%%%%%%%%%%%%%%%%%%%%%%%%%%%%%
%s4.1 ###
%s4.1 #&#
\subsection{Lognormal case}
%%%%%%%%%%%%%%%%%%%%
The lognormal case, that is, when the generator of the $\star$-scale
invariant measure is a Gaussian process, has been entirely treated in
\cite{allez}. Of course, the assumptions are less restrictive
concerning good $\star$-scale invariant measures since their generator
can be entirely described with its two marginals, that is its
covariance function. As a consequence, we do not require
Assumption~\ref{good}, point~(3) in the lognormal case.

%The lognormal case appears in fact quite naturally as we know
%describe. We suppose that we are given a probability space with the
%following structure: a filtration $(\mathcal{F}_{\epsilon})_{\epsilon
%following holds:
%$M_{\epsilon',\epsilon}$ are stationary measures independent from $
%M_{\epsilon',\epsilon} \stackrel{law}{=} \epsilon' M_{\epsilon} (
%For all $\epsilon',\epsilon\leq1$, we have the following
%factorisation\dvtx
%M_{\epsilon}(dx)=e^{\omega_{\epsilon',\epsilon}(x)}M_{\epsilon',
%where $(\omega_{\epsilon',\epsilon}(x))_{x \in\R}$ is adapted to $
%For all $\epsilon$, $x \rightarrow\omega_{\epsilon}(x)$ is cadlag
%almost surely.
%For some $x$, $\epsilon\rightarrow\omega_{\epsilon}(x)$ is
%continuous almost surely
% The above assumptions lead to the lognormal case\dvtx
%For all $\epsilon\leq1$, the process $(\omega_{\epsilon}(x))_{x \in

%s4.2 ###
%s4.2 #&#
\subsection{Reminder about log-ID independently scattered random measures}
%%%%%%%%%%%%%%%%%%%%
$\!\!$The next examples are based on log-ID independently scattered random
measures so that we first collect a few well-known facts about these
measures. The reader is referred to \cite{rosinski} for further
details.

We remind the reader that an ID independently scattered random measure
$\mu$ distributed on a measurable space $(S,\mathcal{B}(S))$ with
control measure $\Gamma$ and kernel $K$~is a collection of random
variables $(\mu(A),A\in\mathcal{B}(S))$ such that:
\begin{longlist}[(2)]
\item[(1)] for every sequence of disjoint sets $(A_n)_n$ in $\mathcal{B}(S)$,
the random variables $(\mu(A_n))_n$ are independent and
\[
\mu\biggl(\bigcup_nA_n \biggr)=\sum
_n\mu(A_n)\qquad\mbox{a.s.};
\]

\item[(2)] for any measurable set $A$ in $\mathcal{B}(S)$, $\mu(A)$ is an ID
random variable whose characteristic function is characterized by
\[
\mathbb E \bigl(e^{iq\mu(A)}\bigr)= \mathbb E \bigl[e^{it \mu(A)}\bigr]=
\exp\biggl(\int_AK(q,s)\Gamma(ds) \biggr).
\]
The control measure $\Gamma$ is a positive $\sigma$-finite measure on
$S$ and the kernel $K$ takes on the form
%e17 ###
%e18 #&#
\begin{equation}
K(q,s)=iqa(s)-\frac{1}{2}q^2\sigma^2(s)+\int
_\mathbb{R}\bigl(e^{iqz}-1-iq\tau(z)\bigr)
\varrho(s,dz),
\end{equation}
where
%e18 ###
%e19 #&#
\begin{equation}
\label{equalone} \bigl|a(s)\bigr|+\sigma^2(s)+\int_\mathbb{R}
\min\bigl(1,z^2\bigr)\varrho(s,dz)=1\qquad\theta\mbox{ a.e.}
\end{equation}
Here $\sigma,a$ belong to $L^\infty(S,\Gamma)$ ($\sigma$ nonnegative)
and $\varrho\dvtx S\times\mathcal{B}(\mathbb{R})\to[0,+\infty]$ is such
that for each fixed $s\in S$, $\varrho(s, dz)$ is a L\'evy measure on
$\mathbb{R}$ and for each $B\in\mathcal{B}(\mathbb{R})$ the function
$\varrho(\cdot,B)$ is measurable and finite whenever $0$ does not
belong to the closure of $B$. The function $\tau$ is any truncation
function. The random measure $\mu$~is characterized by the triple of
measures
$(a(s)\theta(ds),\sigma^2(s)\theta(ds),\varrho(s,dz)\theta(ds))$.
Conversely, to such triple corresponds a unique (in law) ID
independently scattered random measure.
\end{longlist}

%s4.3 ###
%s4.3 #&#
\subsection{Barral--Mandelbrot's type \texorpdfstring{$\star$}{star}-scale invariant
MRMs}\label{barral}
%%%%%%%%%%%%%%%%%%%%

We consider the situation when the dimension $d$ is equal to $1$. We
introduce an ID independently scattered random measure $\mu$
distributed on $(\mathbb{R}\times\mathbb{R}^*_+,\mathcal{B}(\mathbb
{R}\times\mathbb{R}^*_+))$ with
control measure
\[
\Gamma(dt,dy)=dt y^{-2} \,dy
\]
and kernel
\[
K\bigl(q,(t,y)\bigr)=\varphi(q)=i mq-\frac{1}{2}\sigma^2q^2+
\int_{\mathbb
{R}^*} \bigl(e^{iqx} -1 -iqx\mathbf{1}_{|x|\leq1}
\bigr) \nu(dx),
\]
where $\nu(dx)$ is a L\'evy measure on $\mathbb{R}$ and $m,\sigma\in
\mathbb{R}$. We denote by $\psi$ the Laplace exponent associated with
$\varphi$, that is $\psi(q)=\varphi(-iq)$ whenever it makes sense to
consider such a quantity. We assume that $\psi(1)=0$.

We can then define the stationary stochastically continuous ID
process
$(\omega_l(t))_{t\in\mathbb{R}}$ for $l > 0$ by
\[
\omega_l(t) = \mu\bigl(\mathcal{A}_l(t) \bigr),
\]
where $\mathcal{A}_l(t)$ is the triangle like subset
$\mathcal{A}_l(t)=\{(s,y)\in\mathbb{R}\times\mathbb{R}^*_+\dvtx  l \leq
y \leq T, -y/2 \leq t-s \leq y/2\}$, see Figure~\ref{fig1}.

%f1 #&#
\begin{figure}%[h]

\includegraphics{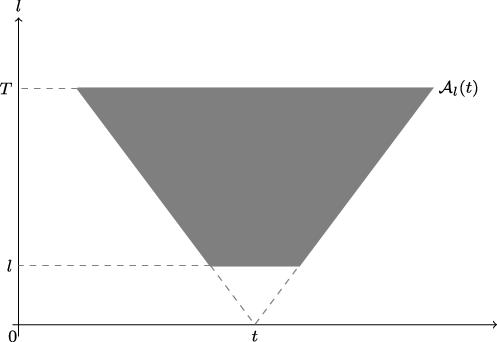}

\caption{Triangle like subset $\mathcal{A}_l(t)$.}\label{fig1}
%%\draw[very thin,color=gray] (-0.1,-0.1) grid (6.1,4.1);
%node[pos=0.5,below]{$t$};
%%\draw[color=red] plot (\x,3*\x/5) node[right] {$\xi$ linear (K41
%theory)};
%%\draw[color=blue] plot (\x,6*\x/5-3/25*\x^2) node[right] {$\xi$ non%linear (K062 theory)};
%%\draw[color=orange] plot (\x,{0.05*exp(\x)}) node[right] {$f(x) =
% \draw[color=gray,-,thick, fill=gray] (1,4) -- (7,4)
%node[right,thick,black]{$\mathcal{A}_l(t)$} -- (4.75,1) -- (3.25,1) --
%(1,4);
% \draw[color=gray,style=dashed] (3.25,1) -- (4,0) -- (4.75,1);
% \draw[color=gray,style=dashed] (3.25,1) -- (0,1)
%node[left,color=black]{$l$};
% \draw[color=gray,style=dashed] (1,4) -- (0,4)
%node[left,color=black]{$T$};
%
\end{figure}

Define now the random measure $M_l$ by $M_l(dt) = e^{\omega_l(t)}
\,dt$. Almost surely, the family of measures $(M_l(dt))_{l>0}$ weakly
converges toward a random measure $M$. When $\psi'(1)-\psi(1)<1$, this
measure is not trivial; see \cite{bacry,Bar}.

Let us check that $M$ is a good $\star$-scale invariant random measure.
Fix \mbox{$\epsilon<1$}, and define the sets $\mathcal {A}_{l,\epsilon
T}(t)=\{(s,y)\dvtx  l \leq y \leq\epsilon T, -y/2 \leq t-s \leq y/2\}$
and $\mathcal{A}_{\epsilon T, T}(t)=\{(s,y)\dvtx  \epsilon T \leq y
\leq T, -y/2 \leq t-s \leq y/2\}$. Note that $\mathcal{A}_l(t) =
\mathcal{A}_{l,\epsilon T}(t) \cup\mathcal{A}_{\epsilon T,T}(t) $ and
that those two sets are disjoint, see Figure~\ref{fig2}. Thus we can write for every
measurable set $A$
%e19 ###
%e20 #&#
\begin{equation}
\label{eql} M_l(A) = \int_A
e^{\omega_{\epsilon T,T}(t)} e^{\omega_{l,\epsilon
T}(t)} \,dt
\end{equation}
with $\omega_{\epsilon T,T}(t) = \mu(\mathcal{A}_{\epsilon T,T}(t))$
and $\omega_{l,\epsilon T}(t)= \mu( \mathcal{A}_{l,\epsilon T}(t))$.

%f2 #&#
\begin{figure}[b]

\includegraphics{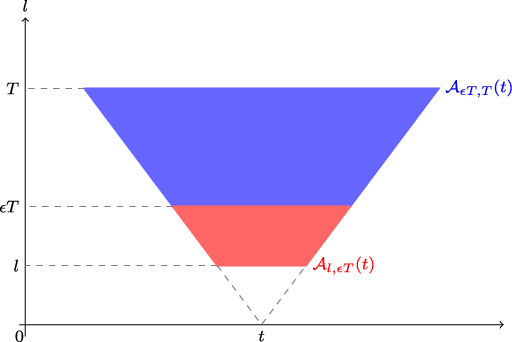}

\caption{Construction of the sets $\mathcal{A}_{\epsilon T,T}(t)$ and
$\mathcal{A}_{l,\epsilon T}(t)$.}\label{fig2}
%%\draw[very thin,color=gray] (-0.1,-0.1) grid (6.1,4.1);
%node[pos=0.5,below]{$t$};
%%\draw[color=red] plot (\x,3*\x/5) node[right] {$\xi$ linear (K41
%theory)};
%%\draw[color=blue] plot (\x,6*\x/5-3/25*\x^2) node[right] {$\xi$ non%linear (K062 theory)};
%%\draw[color=orange] plot (\x,{0.05*exp(\x)}) node[right] {$f(x) =
% \draw[color=blue!60,-,thick, fill=blue!60] (1,4) -- (7,4)
%node[right,thick,blue]{$\mathcal{A}_{\epsilon T,T}(t)$} -- (5.5,2) --
%(2.5,2) -- (1,4);
% \draw[color=red!60,-,thick, fill=red!60] (5.5,2) -- (2.5,2) --
%(3.25,1) -- (4.75,1) node[right,thick,red]{$\mathcal{A}_{l,\epsilon
%T}(t)$} -- (5.5,2);
% \draw[color=gray,thin,style=dashed] (3.25,1) -- (4,0) -- (4.75,1);
% \draw[color=gray,thin,style=dashed] (2.5,2) -- (0,2)
%node[left,color=black]{$\epsilon T$};
% \draw[color=gray,thin,style=dashed] (3.25,1) -- (0,1)
%node[left,color=black]{$l$};
% \draw[color=gray,thin,style=dashed] (1,4) -- (0,4)
%node[left,color=black]{$T$};
%
\end{figure}

We then study equation (\ref{eql}) in the limit $l \to0$; we obtain
%e20 ###
%e21 #&#
\begin{equation}
M(A) = \int_A e^{\omega_{\epsilon T,T}(t)} M^{\epsilon}(dt),
\end{equation}
where $M^{\epsilon}$ is the\vspace*{-2pt} limit when $l \to0$ of the random measure
$M^\epsilon_l(dt)=e^{\omega_{l,\epsilon T}(t)} \,dt$. We easily verify
that $M^\epsilon(\epsilon A) \stackrel{\mathrm{law}}{=} \epsilon M(A)$ writing
%e21 ###
%e22 #&#
\begin{equation}
M^\epsilon_l(\epsilon A) = \epsilon\int
_A e^{\omega_{l,\epsilon
T}(\epsilon t)}\,dt
\end{equation}
and checking that the finite-dimensional marginals of the process
$(\omega_{l,\epsilon T}(\epsilon t))_{t \in\mathbb{R}}$ are the same as
the one of $(\omega_{l,T}(t))_{t\in\mathbb{R}}$; see \cite{Bar}.

By computing the L\'evy exponents of the process $\omega_{\epsilon
T,T}(t)$,
%
%e22 ###
%e23 #&#
\begin{equation}
\mathbb E \bigl[e^{iq_1\omega_{\epsilon
T,T}(t_1)+\cdots+iq_p\omega_{\epsilon
T,T}(t_p)}\bigr]=e^{\psi^\epsilon(q_1,\ldots,q_p,t_1,\ldots,t_p)},
\end{equation}
we obtain
%e23 ###
%e24 #&#
\begin{equation}
\psi^\epsilon(q_1,\ldots,q_p,t_1,\ldots,t_p)=\int_1^{1/\epsilon
}\!\!\!\int
_\mathbb{R}\varphi\Biggl(\sum_{j=1}^pq_jf
\bigl(T_{yt_j}(r)\bigr) \Biggr) \,dr\frac{dy}{y},
\end{equation}
where $f(r)=\mathbf{1}_{[-T/2,T/2]}(r)$ and $T_s\dvtx
t\in\mathbb{R}\mapsto t-s\in\mathbb{R}$ is the usual shift on
$\mathbb{R}$. It is then straightforward to check that $M$ is good
provided that $\int_{z>1}e^{2z}\nu(dz)<+\infty$. We stress that the
L\'evy exponents of the generator, say $X$, are given by
\[
\mathbb E \bigl[e^{iq_1X(t_1)+\cdots+iq_pX(t_p)}\bigr]=\exp\Biggl(\int
_\mathbb
{R}\varphi\Biggl(\sum_{j=1}^pq_jf
\bigl(T_{t_j}(r)\bigr) \Biggr) \,dr \Biggr).
\]

In this example, the $\star$-scale invariance property is easily
understood via the geometric properties of the process, namely the
scaling properties of the cones. Generalizing this example by means of
geometric considerations is far from being obvious and has never been
done in the literature. On the other hand, in view of the results in
this paper, the generalization is straightforward. It suffices to
change the function $f$. To make things more simple, we can, for
instance, choose $f$ equal to any measurable function bounded by $1$
with compact support.

%s4.4 ###
%s4.4 #&#
\subsection{Stable L\'evy chaos}
%%%%%%%%%%%%%%%%%%%%
We focus now on other situations of interest. We consider an infinitely
divisible independently scattered random measure $\mu$ distributed on
$\mathbb{R}$ with the Lebesgue measure $ds$ as control measure and kernel
\[
K(q,t)= \varphi(q)=i mq+\int_{0}^{\infty}
\bigl(e^{-iqx} -1\bigr)\frac
{dx}{x^{1+\alpha}}
\]
for some $\alpha\in\,]0,1[$. Then the associated Laplace exponent is
given by
\[
\forall q\geq0\qquad\psi(q)= mq-\frac{\Gamma(1-\alpha)}{\alpha
}q^\alpha.
\]
Let $(T_t)_{t\in\mathbb{R}}$ be the family of usual shifts on
$\mathbb{R}$. Let $f\dvtx \mathbb{R}\to\mathbb{R}_+$ be any integrable
function with compact support. We define
\[
\|f\|_1=\int_\mathbb{R}f(s) \,ds<+\infty,\qquad\|f
\|_\alpha=\int_\mathbb{R}f(s)^\alpha ds<+
\infty.
\]
We consider the stationary ID random process
\[
\forall t \in\mathbb{R}\qquad X_t=\int f\bigl(T_t(s)
\bigr) \mu(ds).
\]
We have
\[
\mathbb E \bigl[e^{qX_t}\bigr]=e^{\int_\mathbb{R}\psi(qf(s) )
\,ds}=e^{mq\|f\|_1-(\Gamma(1-\alpha)/\alpha)\|f\|_\alpha
q^\alpha}.
\]
So we must set $m=\frac{\Gamma(1-\alpha)\|f\|_\alpha}{\alpha\|f\|
_1}$ to ensure the normalizing condition $\psi(1)=0$. It is obvious to
check that $X$ possesses exponential moments of second order. We assume
that $\psi'(1)<1$, that is,
\[
\|f\|_\alpha<\frac{\alpha}{ \Gamma(2-\alpha)}.
\]

If we consider the L\'evy multiplicative chaos with generator $X$ and
rate function $g(y)=y$, we obtain a nontrivial good star scale
invariant random measure. The scaling factor $\omega_\epsilon$
appearing in (\ref{star}) is a stable ID process.

If we consider the L\'evy multiplicative chaos with generator $X$ with
$f(x)=\mathbf{1}_{B(x,r)}$ ($B(x,r)$ stands for the ball of radius $r$
centered at $x$) and rate function $g(y)=\sum_{n\geq
0}n\mathbf{1}_{[n,n+1[}(y)$, we recover Fan's stable L\'evy chaos.

%%%%%%%%%%%%%%%%%%%%%%%%%%%%%%%%%%%%%%%%%
%%%%%%%%%%%%%%%%%%%%%%%%%%%%%%%%%%%%%%%%%
%s5 ###
%s5 #&#
\section{Conjectures and open problems}
%%%%%%%%%%%%%%%%%%%%%%%%%%%%%%%%%%%%%%%%%
%%%%%%%%%%%%%%%%%%%%%%%%%%%%%%%%%%%%%%%%%

%s5.1 ###
%s5.1 #&#
\subsection{Convergence of the derivative martingale}
%%%%%%%%%%%%%%%%%%%%%%%%%%%%%%%%%%%%%
Consider a L\'evy multiplicative chaos on $\mathbb{R}^d$ with integrating
measure the Lebesgue measure. Let $\psi$ be the Laplace exponent to its
generator and $(X^\epsilon)_\epsilon$ the associated approximating
family. Consider the martingale
%e24 ###
%e25 #&#
\begin{equation}
\label{martheta} M_\epsilon^\theta(dx)=\int_\cdot
e^{\theta X_\epsilon(x)-\psi
(\theta)\ln(1/\epsilon)} \,dx.
\end{equation}
For $\theta\in\mathbb{R}_+$, define the function
\[
\forall q \in\mathbb{R}\qquad\xi_\theta(q)=\bigl(d+\psi(\theta)
\bigr)q-\psi(q \theta).
\]
There is at most one solution $\theta_0>0$ to the equation $\xi
'_\theta(1)=0$.

Let us discuss the (only nontrivial) situation when
$0<\theta_0<\infty$. The nondegeneracy condition of Theorem \ref{a_mont}
reads $\xi'_\theta(1)>0 $, and it is valid if and only if
$\theta<\theta_0$. Inspired by the Gaussian case (see \cite{cfKah}) or
multiplicative cascades (see \cite{cfKahPey}), prove the following:
%
%co18 #&#
\begin{conjecture}
For $\theta\geq\theta_0$, the martingale defined by (\ref{martheta})
converges almost surely toward $0$.
\end{conjecture}

Deduce the following estimate for the statistics of the maximum of a
log-correlated infinitely divisible process:
%
%co19 #&#
\begin{conjecture}
For any open bounded set $A$, we have
\[
\max_{x\in A}X_\epsilon(x)-\psi'(
\theta_0)\ln\frac{1}{\epsilon
}\to-\infty\qquad\mbox{as }\epsilon\to0.
\]
\end{conjecture}

At the critical point $\theta=\theta_0$, we must introduce the
so-called derivative martingale (see \cite{Rnew7})
%e25 ###
%e26 #&#
\begin{equation}
\label{derivmartheta} D_\epsilon^\theta(dx)=\int_\cdot
\biggl(\psi(\theta_0)\ln\frac
{1}{\epsilon}-X_\epsilon(x)
\biggr)e^{\theta_0 X_\epsilon(x)-\psi(\theta
_0)\ln(1/\epsilon)} \,dx.
\end{equation}

%co20 #&#
\begin{conjecture}
Prove that the derivative martingale almost surely converges toward a
nontrivial positive random measure, denoted by $M'$. Prove that $M'$
does not possess atoms.
\end{conjecture}

%co21 #&#
\begin{conjecture}
Prove that the derivative martingale can be obtained as a suitable
renormalization of the sequence $(M^{\theta_0}_\epsilon)_\epsilon$.
More precisely,
\[
\sqrt{\ln\frac{1}{\epsilon}}M^{\theta_0}_\epsilon(dx)\to c
M'(dx)\qquad\mbox{as }\epsilon\to0
\]
for some deterministic factor $c$.
\end{conjecture}
%
%s5.2 ###
%s5.2 #&#
\subsection{Conjectures about star scale invariance}
%%%%%%%%%%%%%%%%%%%%%%%%%%%%%%%%%%%%%
Consider the star scale invariance equation in great generality, that
is:
%
%de22 #&#
\begin{definition}[(Star scale invariance)] \label{defgeneral}
A random Radon measure $M$ is star scale invariant if for all
$0<\epsilon\leq1$, $M $ obeys the cascading rule
%e26 ###
%e27 #&#
\begin{equation}
\label{starc} \bigl(M (A) \bigr)_{A\in\mathcal{B}(\mathbb
{R}^d)}\stackrel{\mathrm{law}}{=} \biggl(\int
_Ae^{ \omega_{\epsilon}(r)}M^\epsilon(dr)
\biggr)_{A\in
\mathcal{B}(\mathbb{R}^d)},
\end{equation}
where $\omega_{\varepsilon}$ is a stationary stochastically continuous
Gaussian process, and $M^{ \varepsilon}$ is a~random measure
independent from $\omega_{\varepsilon}$ satisfying the scaling relation
%e27 ###
%e28 #&#
\begin{equation}
\label{star1c} \bigl(M^{\epsilon}(A) \bigr)_{A\in\mathcal{B}(\mathbb
{R}^d)}\stackrel{\mathrm{law}} {=}
\biggl(M \biggl(\frac{A}{\epsilon}\biggr) \biggr)_{A\in
\mathcal{B}(\mathbb{R}^d)}.
\end{equation}
\end{definition}

Observe that the main difference with (\ref{star}) is that we do not
impose here the normalization $ \mathbb E [e^{
\omega_{\epsilon}(r)}]=\epsilon^d$. As soon as the measure possesses a
moment of order $1+\delta$ for some $\delta>0$, the condition $ \mathbb
E [e^{ \omega_{\epsilon}(r)}]=\epsilon^d$ must be satisfied. So it
remains to investigate situations when the measure possesses moments of
at most order $1$.

Inspired by the discrete multiplicative cascade case (see
\cite{durrett}), we conjecture:
%
%co23 #&#
\begin{conjecture}
Prove that, if $M$ is a good star scale invariant measure, there exists
a $\alpha\in\,]0,1]$ such that
\[
\mathbb E \bigl[e^{\alpha\omega_{\epsilon}(r)}\bigr]=\epsilon^{ d}.
\]
\end{conjecture}

Assuming this, we may follow the proof of Theorem \ref{thuniq} to see
that the process $\alpha\omega_{\epsilon}$ has a structure given by
(\ref{version2}).\vadjust{\goodbreak} More precisely, we can then rewrite the process
$\omega_\epsilon$ as
%e28 ###
%e29 #&#
\begin{equation}
\label{structc} \omega_{\epsilon}(x)=\frac{1}{\alpha}X_\epsilon(x)-
\frac{ \psi
(1)}{\alpha}\ln\frac{1}{\epsilon}-\frac{d}{\alpha}\ln\frac
{1}{\epsilon}
\end{equation}
for some family $(X_\epsilon)_\epsilon$ of the type (\ref{version2}),
and $\psi$ is its Laplace exponent.

%co24 #&#
\begin{conjecture}
Assume that $M$ is an ergodic good star scale invariant measure, and
let (\ref{structc}) be the decomposition of $\omega_\epsilon$:
\begin{longlist}[(3)]
\item[(1)] If $\alpha=1$ and $\psi'(1)-\psi(1)<d$, then the law of
the solution $M$ is Levy multiplicative chaos multiplicative
chaos with rate function $g(y)=y$ up to a deterministic
multiplicative constant; see Theorem \ref{thuniq}.

\item[(2)] If $\alpha=1$ and $\psi'(1)-\psi(1)=d$, prove that the
law of the solution $M$ is that of the limit of the derivative
martingale, namely $M'$ described above, up to a multiplicative
constant.

\item[(3)] If $\alpha<1$ and $\psi'(1)-\psi(1)<d$, prove that $M$
    is an ``atomic L\'evy multiplicative chaos'' (see \cite{Rnew4}
    in the Gaussian case) up to a multiplicative constant. More
    precisely, the law can be constructed as follows:
\begin{enumerate}[(a)]
\item[(a)] Sample the standard Levy multiplicative chaos
\[
\widebar{M}(dx)=\lim_{\epsilon\to0}\int_\cdot
e^{
X_\epsilon(x)-\psi(1)\ln(1/\epsilon)} \,dx.
\]
The measure $\widebar{M}$ is perfectly defined since
$\psi'(1)-\psi(1)<d$.

\item[(b)] Sample a point process $M$ whose law, conditioned on
    $\widebar{M}$, is that of an independently scattered random
    measure characterized by
\[
\forall q\geq0\qquad\mathbb E \bigl[e^{-qM(A)}|\widebar{M}
\bigr]=e^{-q^\alpha\widebar{M}(A)}.
\]
\end{enumerate}

\item[(4)] If $\alpha<1$ and $\psi'(1)-\psi(1)=d$, prove that $M$
is an atomic L\'evy multiplicative chaos of a second type. More
precisely, the law can be constructed as follows:
\begin{enumerate}[(a)]
\item[(a)] sample the derivative L\'evy multiplicative chaos
$M'(dx)$ as described above;

\item[(b)] sample a point process $M$ whose law, conditioned on
$M'$, is that of an independently scattered random measure
characterized by
\[
\forall A\in\mathcal{B}\bigl(\mathbb{R}^d\bigr),\ \forall q\geq0 \qquad
\mathbb E \bigl[e^{-qM(A)}|M'\bigr]=e^{-q^\alpha M'(A)}.
\]
\end{enumerate}
\end{longlist}
\end{conjecture}

The reader may find in \cite{Rnew4,Rnew7} some further conjectures in
the Gaussian case that can also be adapted to this framework. In
particular, adapting these conjectures to our framework, the reader may
deduce conjectures about the glassy phase and freezing phenomena of
log-correlated infinitely divisible random potentials and about the
asymptotics of the extreme values of log-correlated infinitely
divisible random fields.

\begin{appendix}
\section{\texorpdfstring{Proof of Lemma \lowercase{\protect\ref{lemme_id}}}{Proof of Lemma 1}}

We first state the following intermediate lemma:

%le25 #&#
\begin{lemma}\label{divmeas}
Let $(F(x))_{x \in\mathbb{R}^d}$ and $(G(x))_{x \in\mathbb{R}^d}$
be two stationary and nonnegative stochastically continuous processes.
We consider a nontrivial stationary random measure $\eta$ on $\mathbb
{R}^d$ independent of $F,G$. We suppose that there exists $\gamma>0$
such that $ \mathbb E [F(x)^{1+\gamma}]<\infty$, $ \mathbb E
[G(x)^{1+\gamma}]<\infty$, and $ \mathbb E [ \eta(K)^{\gamma}] <
\infty$ for all compact set $K$. If the following equality on measures holds:
\[
F(x) \eta(dx) \stackrel{\mathrm{law}} {=} G(x) \eta(dx),
\]
then the two processes $F$ and $G$ have same law.
\end{lemma}

\begin{pf} We consider the case $d=1$ (the higher dimensions work the same).
Let $\delta>0$. Notice that $ \mathbb E [\eta([0,\delta])^\alpha]>0$
for all $\alpha\in\,]0, \gamma[$. Indeed, the measure is stationary and
nontrivial. Choose now $\alpha\in\,]0, \min(\gamma,1)[$. Notice that the
mapping $x\in\mathbb{R}_+\mapsto x^\alpha$ is sub-additive. Therefore
$|x^\alpha-y^\alpha|\leq|x-y|^\alpha$ for
any $x,y\geq0$. We deduce the following inequality:% for all $\alpha
\begin{eqnarray*}
&&\biggl| \mathbb E \biggl[ \biggl( \int_0^{\delta} F(x)
\eta(dx) \biggr) ^{\alpha} \biggr] - \mathbb E \biggl[ \biggl( \int
_0^{\delta} F(0) \eta(dx) \biggr) ^{ \alpha}
\biggr] \biggr|
\\[-2.5pt]
&&\qquad \leq \mathbb E \biggl[ \biggl| \int_0^{\delta}
F(x) \eta(dx)- \int_0^{\delta} F(0) \eta(dx)
\biggr|^{\alpha} \biggr]
\\[-2.5pt]
&&\qquad \leq \mathbb E \biggl[ \biggl( \int_0^{\delta} \bigl|
F(x)-F(0)\bigr| \eta(dx) \biggr)^{\alpha} \biggr].
\end{eqnarray*}
The mapping $x\in\mathbb{R}_+\mapsto x^\alpha$ is concave. So we use
Jensen's inequality applied to $ \mathbb E [\cdot | \eta]$, and we get:
\begin{eqnarray*}
&& \biggl| \mathbb E \biggl[ \biggl( \int_0^{\delta} F(x)
\eta(dx) \biggr) ^{\alpha} \biggr] - \mathbb E \biggl[ \biggl( \int
_0^{\delta} F(0) \eta(dx) \biggr) ^{ \alpha}
\biggr] \biggr|
\\[-2.5pt]
&&\qquad  \leq\mathbb E \biggl[ \biggl( \int_0^{\delta}
\mathbb E \bigl[ \bigl| F(x)-F(0)\bigr| \bigr] \eta(dx) \biggr)^\alpha\biggr]
\\[-2.5pt]
&&\qquad \leq\sup_{x \in[0, \delta]} \mathbb E \bigl[ \bigl| F(x)-F(0)\bigr|
\bigr]^\alpha\mathbb E \bigl[ \eta[0,\delta]^{\alpha} \bigr].
\end{eqnarray*}
Since $\sup_{x \in[0, \delta]} \mathbb E [ | F(x)-F(0)| ]
\mathop{\rightarrow}\limits_{\delta\to0} 0$, we get that
\[
\frac{ \mathbb E [ ( \int_0^{\delta} F(x) \eta(dx)
) ^{ \alpha} ] }{ \mathbb E [ \eta[0,\delta
]^{\alpha} ]} \mathop{\rightarrow}\limits
_{\delta\to0} \mathbb E
\bigl[F(0)^{ \alpha}\bigr].
\]
Similarly, we get the above convergence with $F$ replaced by $G$: this
shows that $F(0)$ and $G(0)$ have the same distribution. We show
similarly, for all $x_1, \ldots, x_n$, that $(F(x_1), \ldots, F(x_n))$
and $(G(x_1), \ldots, G(x_n))$ have the same distribution.
\end{pf}

%
%We consider the case $d=1$ (the higher dimensions work the same).
%Let $\epsilon>0$. We have the following inequality for all $\alpha\in
%]0, \gamma[$ (recall that $\E[|.|^{1+\alpha}]^{\frac{1}{1+\alpha}}$ is
%a norm):
%| \E\left[ \left( \int_0^{\epsilon} F(x) \eta(dx) \right) ^{1+
%& \leq\E\left[ | \int_0^{\epsilon} F(x) \eta(dx)- \int_0^{\epsilon}
%F(0) \eta(dx) |^{1+\alpha} \right]^{\frac{1}{1+\alpha}} \\
%& \leq\E\left[ | \int_0^{\epsilon} | F(x)-F(0)| \eta(dx) |^{1+
%& \leq\E\left[ \int_0^{\epsilon} \E[ | F(x)-F(0)|^{1+\alpha} ]
%]^{\frac{1}{1+\alpha}}, \\
%& \leq\sup_{x \in[0, \epsilon]} \E[ | F(x)-F(0)|^{1+\alpha} ]^{
% \end{eqnarray*}
%where we use Jensen applied to $\E[. | \eta]$. Since $\sup_{x \in
%[0, \epsilon]} \E[ | F(x)-F(0)|^{1+\alpha} ] \underset{\epsilon\to0}{
%Similarly, we get the above convergence with $F$ replaced by $G$: this
%shows that $F(0)$ and $G(0)$ have the same distribution. We show
%similarly, for all $x_1, \cdots, x_n$, that $(F(x_1), \cdots, F(x_n))$
%and $(G(x_1), \cdots, G(x_n))$ have the same distribution.

Now, we can finish the proof of Lemma \ref{lemme_id}:

\begin{pf*}{Proof of Lemma \ref{lemme_id}}
By iterating (\ref{star}) and using the above lemma, the process
$(\omega_{\epsilon}(x))_{x \in\mathbb{R}^d}$ is such that ($\epsilon,
\epsilon' < 1$),
%e29 ###
%e30 #&#
\begin{equation}
\label{eqcascade} \bigl(\omega_{\epsilon\epsilon'}(x)\bigr)_{x \in
\mathbb{R}^d}\stackrel
{\mathrm{law}} {=} \biggl(\omega_{\epsilon}(x)+\tilde{\omega}_{\epsilon
'}
\biggl(\frac
{x}{\epsilon}\biggr)\biggr)_{x \in\mathbb{R}^d},
\end{equation}
where $\omega_{\epsilon}$ and $\tilde{\omega}_{\epsilon'}$ are
independent copies of $\omega_{\epsilon}$ and $\omega_{\epsilon'}$. We
fix $\epsilon$ and consider $\epsilon_n={\epsilon}^{1/n}$. Of
course $\epsilon_n^n=\epsilon$. By iterating the cascade rule
(\ref{eqcascade}), we get
\[
\bigl(\omega_{\epsilon}(x)\bigr)_{x \in\mathbb{R}^d}\stackrel{\mathrm{law}} {=}
\Biggl(
\sum_{k=0}^{n-1}\omega_{\epsilon_n}^{(k)}
\biggl(\frac{x}{\epsilon_n^k}\biggr)\Biggr)_{x
\in\mathbb{R}^d},
\]
where the $\omega_{\epsilon_n}^{(k)}$ are independent processes of law
$\omega_{\epsilon_n}$. Fix $x,y \in\mathbb{R}^d$. We therefore have
for all
$\lambda, \mu$,
\[
\lambda\omega_{\epsilon}(x)+\mu\omega_{\epsilon}(y)\stackrel{\mathrm{law}} {=}
\sum_{k=0}^{n-1} \mu\omega_{\epsilon_n}^{(k)}
\biggl( \frac
{x}{\epsilon_n^k} \biggr)+\lambda\omega_{\epsilon_n}^{(k)}
\biggl(\frac
{y}{\epsilon_n^k}\biggr).
\]
The stochastic continuity of the process $\omega$ with respect to
$\epsilon$ entails, for all $\eta>0$,
\[
\sup_{0 \leq k \leq n-1} P\biggl(\biggl| \mu\omega_{\epsilon_n}^{(k)}
\biggl( \frac
{x}{\epsilon_n^k} \biggr)+\lambda\omega_{\epsilon_n}^{(k)}
\biggl(\frac
{y}{\epsilon_n^k}\biggr) \biggr|> \eta\biggr) \mathop{\rightarrow}\limits
_{n \to
\infty} 0.
\]
By a classical theorem on independent triangular arrays (see Chapter
XVII in \cite{Feller}), this shows that the couple $(
\omega_{\epsilon}(x), \omega_{\epsilon}(y)) $ is ID. One proceeds
similarly to show that, for all $(x_1, \ldots, x_n)$, the vector
$(\omega_{\epsilon}(x_1), \ldots,\omega_{\epsilon}(x_n)) $ is ID.
\end{pf*}

%s7 ###
%s7 #&#
\section{\texorpdfstring{Proof of Theorem \lowercase{\protect\ref{a_mont}}}{Proof of Theorem 4}}
%%%%%%%%%%%%%%%%%%%%%%%%%%%%%%%%%%%%%%%%%
We adapt the proofs of \cite{cfKah,rhovar}.

\textit{The class} $R_{\alpha}$.
Let $B$ be a nonempty ball of $\mathbb{R}^{d}$.
We introduce the set $R_{\alpha}$ of Radon measures $\nu$ on $B$
satisfying: for any $\varepsilon> 0$, there exist $\delta> 0, D > 0$
and a~compact set $K_{\varepsilon} \subset B$ with $\nu(B \setminus
K_{\varepsilon}) < \varepsilon$ such that the measure $\nu
_{\varepsilon} = \mathbf{1}_{K_{\varepsilon}}(x) \nu(dx)$
satisfies, for every
open set $U\subset B$,
%
%e31 #&#
\begin{equation}
\label{inegnu} \nu_{\varepsilon}(U) \leq D \times \operatorname{diam}(U)^{\alpha+
\delta}.
\end{equation}
We further define the set of Radon measures $R^{\alpha}_{-}= \bigcap
_{\beta< \alpha} R^{\beta}$. For a Radon measure~$\nu$, we define the
quantity
\[
C_{\alpha}(\nu) = \int_{B \times B} \frac{1}{|x-y|^{\alpha}} \nu
(dx) \nu(dy).
\]
It is plain to see that
\[
C_{\alpha}(\nu) < \infty\quad\Longrightarrow\quad\nu\in R^{\alpha}_{-}.
\]
Conversely, a measure obeying (\ref{inegnu}) satisfies $C_\beta(\nu
)<+\infty$ for all $\beta<\alpha+\delta$.

We show the following intermediate result:
%
%le26 #&#
\begin{lemma}\label{Int}
Consider a Radon measure $\kappa\in R_{\alpha}$. Let $N$ be the Radon
measure defined on $B$ by
\[
N(dx) = \lim_{\varepsilon\searrow0} e^{X^{\varepsilon}_{x} - \psi
(1)\ln(1/\varepsilon)}\kappa(dx) =: \lim
_{\varepsilon\searrow0} N_{\varepsilon}(dx).
\]
If $\widebar{F} < \alpha$, then the martingale $(N_\epsilon
(B))_\epsilon$ is regular and $N \in R_{\alpha- \psi'(1) + \psi(1)}$.
\end{lemma}

% Note that by stationarity, the condition $C_{\gk,f,\mu}$ is
%equivalent to the following:
% \begin{equation}\label{crit2}
% \sup_{ \epsilon> 0} \int_{B \times B} e^{\int_{\R^{d} \times[1,1/
% \int_{\R} (e^{izf(u(y-x),s)}-1)(e^{izf(0,s)}-1) \rho(s,dz) \right)
% \end{equation}

\begin{pf}
We first show that the martingale $(N_{\epsilon }(B))_{\epsilon}$ is
regular. For this, we use the fact that $F(\cdot)$ verifies
Assumption~\ref{vitg} to get (for some positive constant
$S=\sup_{\mathbb{R}_+\times\mathbb{R}^d}|h|$)
\begin{eqnarray*}
\mathbb E \bigl[N_{\epsilon}(B)^{2}\bigr] &=& \int
_{B \times B} \mathbb E \bigl[e^{X_{x}^{\epsilon}+ X_{y}^{\epsilon}}
\bigr]
e^{-2\psi(1)\ln(1/\epsilon)} \kappa(dx) \kappa( dy)
\\
&=& \int_{B \times B} e^{\int_{1}^{1/\varepsilon}F (g(u)(x-y)
)((du)/u)} \kappa(dx) \kappa(dy)
\\
& \leq& \int_{B \times B} e^{\int_{1}^{\infty}\bigl|F (g(u)(x-y)
)\bigr|((du)/u)} \kappa(dx) \kappa(dy)
\\
& \leq& \int_{B \times B} e^{\widebar{F}\ln_+(1/|x-y|)+S} \kappa
(dx) \kappa(dy)
\\
& \leq& e^S \int_{B \times B}\max\biggl(
\frac{1}{|x-y|^{\overline
{F}}},1 \biggr)\kappa(dx) \kappa(dy)
\end{eqnarray*}
and the last integral is finite as soon as $\widebar{F} < \alpha$.
Hence, the martingale $(N_\epsilon(B))_\epsilon$ is regular.

We consider a compact set $K\subset B$. Even if it means multiplying
$\kappa$ by a positive constant, we assume that $\kappa(K) = 1$. We
consider on $\Omega\times K$ the probability measure $\mathbb{Q}$
defined by
\[
\int_{\Omega\times K} f(\omega,x) \,d\mathbb{Q}= \mathbb E \biggl[ \int
_{K} f(\omega,x) N(dx) \biggr],
\]
where $f$ is any nonnegative measurable function.

For $0 < \varepsilon' < \varepsilon< 1$, we define the process
$(X^{\varepsilon',\varepsilon}_{x})_{x \in\mathbb{R}^{d}}$ by
\[
\forall x \in\mathbb{R}^{d}\qquad X^{\varepsilon',\varepsilon}_{x} =
X^{\varepsilon'}_{x} - X^{\varepsilon}_{x} - \psi(1)\ln
\bigl(\varepsilon/\varepsilon'\bigr).
\]
Because of expression (\ref{defX}), it is straightforward to check
that, given $\varepsilon_{1} < \varepsilon_{2} < \cdots< \varepsilon
_{n}$, the processes
$X^{\varepsilon_{1},\varepsilon_{2}},X^{\varepsilon_{2},\varepsilon
_{3}}, \ldots, X^{\varepsilon_{n-1},\varepsilon_{n}} $ are $\mathbb
{Q}$-independent. Moreover, for $\lambda\geq0$ and because
$(N_\epsilon)_\epsilon$ is uniformly integrable, we have
% \ln\frac{1}{\epsilon}+\int_1^{1/\epsilon}\int_{\R^d}\cos(tg(y)
\begin{eqnarray*}
\hspace*{-5pt}&&\int e^{\lambda X^{\varepsilon',\varepsilon}_{x}} \,d \mathbb{Q}
\\
\hspace*{-5pt}&&\!\!\qquad = \int_{K} \mathbb E \biggl[\exp\biggl\{\lambda\int_{1/\epsilon}^{1/\epsilon'}\!\!\! \int_{\mathbb{R}^d}\cos\bigl(xg(y)\cdot u\bigr)
W(du,dy)
\\
\hspace*{-5pt}&&\!\!\hspace*{79pt}{} +\sin\bigl(xg(y)\cdot u\bigr) W'(du,dy) +\lambda b\ln\frac{\epsilon
}{\epsilon'}- \lambda\psi(1)\ln\bigl(\varepsilon/\varepsilon'\bigr) \biggr\}
\\
\hspace*{-5pt}&&\!\!\hspace*{57pt}{}\times \exp\biggl\{\lambda\int_{1/\epsilon}^{1/\epsilon'}\!\!\!\int_Sf\bigl(T_{tg(y)}(s)\bigr)
\\
\hspace*{-5pt}&&\!\!\hspace*{130pt}{}\times \biggl[N(ds,dy)
\\
\hspace*{-5pt}&&\!\!\hspace*{145pt}{} -\bigl(1\vee
\bigl|f\bigl(T_{tg(y)}(s)\bigr)\bigr|\bigr)^{-1}\theta(ds)\frac{dy}{y}\biggr]\biggr\}
\\
\hspace*{-5pt}&&\!\!\hspace*{220pt}{} \times \exp\biggl\{ X^{\epsilon
'}_x\hspace*{-2pt}-\psi(1)\ln\frac{1}{\epsilon'}\biggr\} \biggr] \kappa(dx)
\\
\hspace*{-5pt}&&\!\!\qquad = \int_{K} \mathbb E \biggl[\exp\biggl\{(\lambda+1)\int_{1/\epsilon
}^{1/\epsilon'}\!\!\! \int_{\mathbb{R}^d}\cos\bigl(xg(y)\cdot u\bigr)
W(du,dy)
\\
\hspace*{-5pt}&&\!\!\hspace*{80pt}{}+\sin\bigl(xg(y)\cdot u\bigr) W'(du,dy) +(\lambda+1) b\ln\frac
{\epsilon}{\epsilon'}
\\
\hspace*{-5pt}&&\!\!\hspace*{174pt}{} - (\lambda+1)\psi(1)\ln\bigl(\varepsilon
/\varepsilon'\bigr) \biggr\}
\\
\hspace*{-5pt}&&\!\!\hspace*{57pt}{}\times \exp\biggl\{(\lambda+1)\int_{1/\epsilon
}^{1/\epsilon'}\!\!\!\int_Sf\bigl(T_{tg(y)}(s)\bigr)
\\
\hspace*{-5pt}&&\!\!\hspace*{155pt}{}\times \biggl[N(ds,dy)
\\
\hspace*{-5pt}&&\!\!\hspace*{170pt}{} -\bigl(1\vee
\bigl|f\bigl(T_{tg(y)}(s)\bigr)\bigr|\bigr)^{-1}\theta(ds)\frac{dy}{y}\biggr]\biggr\} \biggr] \kappa(dx)
\\
\hspace*{-5pt}&&\!\!\qquad = \exp\biggl\{\psi(\lambda+1) \ln\bigl(\varepsilon/\varepsilon'\bigr)-(\lambda
+1)\psi(1)\ln\bigl(\varepsilon/\varepsilon'\bigr)\biggr\}.
\end{eqnarray*}
In particular, under $\mathbb{Q}$, the process $u \in\mathbb{R}^{+}
\mapsto
X^{e^{-u},1}$ is an integrable L\'{e}vy process. Thus from the strong
law of large numbers, we get that $\mathbb{Q}$-almost surely:
\[
\frac{X^{e^{-u},1}}{u} \to\psi'(1)-\psi(1),
\]
when $u \to\infty$. Consequently, $ \mathbb P $ almost surely,
%e31 ###
%e32 #&#
\begin{equation}
\label{conv} N\mbox{ a.s.},\qquad\frac{X^{e^{-u}}_{x}}{u} \to\psi'(1).
\end{equation}
In particular, by Egoroff's theorem, there exists a compact set
$K_{\varepsilon}^{1} \subset K$ such that $N(K\setminus K_{\varepsilon
}^{1}) < \varepsilon$ and the\vspace*{-2pt} convergence (\ref{conv}) is uniform with
respect to $x \in K_{\varepsilon}^{1}$. Let now $q > 0$, and define
$N_{q}(dy) = \lim_{\epsilon\searrow0} e^{X^{\epsilon,e^{-q}}_{y}}
\kappa(dy)$ and $P_{q}(x) = N_{q}(B_{x}^{q} \cap K)$ where $B_{x}^{q}$
denotes the ball centered on $x$ and with radius $e^{-q}$. We finally
define the function
\[
\theta_{q}(x,y) = \mathbf{1}_{\{|x-y| \leq e^{-q}\}},
\]
in such a way that $P_{q}(x) = \int_{K} \theta_{q}(x,y) N_{q}(dy)$.
Thus we have:
\begin{eqnarray*}
\int P_{q} \,d\mathbb{Q}&=& \mathbb E \biggl[ \int_{K \times K}
\theta_{q}(x,y) N_{q}(dx) N(dy) \biggr]
\\
&=& \lim_{\epsilon\to0} \mathbb E \biggl[ \int_{K \times K}
\theta_{q}(x,y) e^{X^{\epsilon,e^{-q}}_x+X^{\epsilon,e^{-q}}_y}\kappa
(dx) \kappa(dy) \biggr]
\\
&=& \int_{K \times K} \theta_{q}(x,y) e^{ \int_{ [e^{q},\infty]}
F(g(u)(y-x)) ((du)/u)}
\kappa(dx) \kappa(dy). % &= \lim_{\gep\searrow0} \int_{K \times K}
% L\left(f(u(y-x),s) + f(0,s),s \right) \theta(ds) \frac{du}{u} } e^{- 2
% & \leq C \int_{K \times K} \theta_{q}(x,y) e^{-\int_{\R^{d} \times
%[1,e^{q}]} L\left(f(u(y-x),s) + f(0,s),s \right)
% \theta(ds) \frac{du}{u} } e^{ 2q\psi(1)} \gk(dx) \gk(dy)
\end{eqnarray*}
%
% where in the last inequality we made use of condition $C_{\gk,f,\mu}$.

Let $\beta>\widebar{F}$ be fixed. By using Assumption~\ref{vitg} and
the above relation, we obtain (for some positive constant
$S=\sup_{\mathbb{R}_+\times\mathbb{R}^d}|h|$)
\begin{eqnarray*}
&& \int\sum_{n \geq1} e^{\beta n} P_{n}
\,d\mathbb{Q}
\\
&&\qquad  = \sum_{ n \geq 1} \int_{K \times K}
\theta_{n}(x,y) e^{\beta n} e^{ \int_{ [e^{n},\infty]} F(g(u)(y-x))
((du)/u)} \kappa(dx)
\kappa(dy)
\\
&&\qquad  = \int_{K \times K} \sum_{ 1 \leq n \leq-\ln(|x-y|)}
e^{\beta n} e^{ \int_{ [e^{n},\infty]} F(g(u)(y-x)) ((du)/u)} \kappa
(dx) \kappa(dy)
\\
&&\qquad \leq e^S \int_{K \times K} \sum
_{ 1 \leq n \leq-\ln(|x-y|)} e^{\beta n} e^{ \widebar{F}\ln(1/(e^n|x-y|))} \kappa(dx)
\kappa(dy)
\\
&&\qquad \leq e^S\int_{K \times K} \sum
_{ 1 \leq n \leq-\ln(|x-y|)} e^{(\beta-\widebar{F}) n} \frac
{1}{|x-y|^{\widebar{F}}}\kappa(dx)
\kappa(dy).
\end{eqnarray*}
Note that, for some positive constant $D$,
\[
\sum_{ 1 \leq n \leq-\ln(|x-y|)} e^{(\beta-\widebar{F}) n}\leq D
\frac{1}{|x-y|^{\beta-\widebar{F}}},
\]
in such a way that
\begin{eqnarray*}
\int\sum_{n \geq1} e^{\beta n} P_{n}
\,d\mathbb{Q} & \leq& De^S\int_{K \times K}
\frac{1}{|x-y|^{\beta}}\kappa(dx) \kappa(dy) =DD'C_\beta(
\kappa).
\end{eqnarray*}
The last term is finite as soon as $\beta< \alpha$. Thus for $\beta \in\,]\widebar{F}, \alpha[$, $\mathbb{Q}$ a.s., $ e^{\beta n} P_{n} \to0$ as
$n \to\infty$. In particular, one can find a compact set
$K_{\varepsilon}^{2} \subset K$ such that $N(K \setminus
K_{\varepsilon}^{2}) < \varepsilon$ and such that, $N$ almost surely,
\[
\limsup_{n \to\infty} \frac{\log(P_{n}(x))}{n} \leq- \beta
\]
uniformly for $x \in K_{\varepsilon}^{2}$. Setting $\widetilde{K} =
K_{\varepsilon}^{2} \cap K_{\varepsilon}^{1}$ and
$N_{\widetilde{K}} = \mathbf{1}_{\widetilde{K}}(x)N(dx)$, we get that,
uniformly with respect to $x \in\widetilde{K}$,
\begin{eqnarray*}
\limsup_{n \to\infty} \frac{\log(N_{\widetilde{K}}(B_{n}^{x}))}{n} & =&
\limsup
_{n \to\infty} \frac{\log(\int_{\widetilde{K}
\cap
B_{n}^{x}} e^{X^{e^{-n}}_{u} - \psi(1)n} N_{n}(du) )}{n}
\\
& \leq& - \beta+ \psi'(1)-\psi(1).
\end{eqnarray*}
This entails in particular that $M \in R_{\alpha- \psi'(1)+\psi
(1)}$.
\end{pf}
%
% This last quantity is finite as soon as the following is verified:
% \[
% \gb< \liminf_{n \to\infty}\frac{1}{n} \int_{\R^{d} \times[1,e]}
%L(f(0,s) + f(e^{n}u(x-y),s),s)
% \theta(ds) \frac{du}{u} - 2 \psi(1).
% \]

Making use of Lemma \ref{Int}, we now prove Theorem \ref{a_mont}.

\begin{pf*}{Proof of Theorem \ref{a_mont}}
The basic idea is to show that a L\'evy multiplicative chaos satisfying
$\psi'(1) - \psi(1) < d$ can be decomposed as an iterated L\'evy
multiplicative chaos.

First, fix an integer $n$ such that
\[
\widebar{F}< n\bigl(d-\psi'(1)+\psi(1)\bigr).
\]
There exist $n$ independent identically distributed approximating
families $(X^{(1),\epsilon}, \ldots,X^{(n),\epsilon})_{\epsilon\in]0,1[}$, respectively,\vspace*{1pt} associated with
$(S,W^{(i)},W^{\prime (i)},N^{(i)},\break R/n, \theta/n,f, (T_x)_x )$ where the
$(W^{(i)},W^{\prime (i)},N^{(i)})_{1\leq i\leq n}$ are all\vspace*{1pt} independent.
We assume that the triples $(W^{(1)},W^{\prime (1)},N^{(1)}),\ldots,
(W^{(n)},W^{\prime (n)},N^{(n)})$ are, respectively, constructed on
the probability space $(\Omega_{1}, \mathbb P ^{1}),
\ldots,(\Omega_{n}, \mathbb P ^{n}) $, and we define $\Omega=
\Omega_{1} \times\cdots\times\Omega_{n}$ equipped with the probability
measure $ \mathbb P = \mathbb P ^{1} \otimes\cdots\otimes\mathbb P
^{n}$.

We define recursively for $1 \leq k \leq n$,
%e32 ###
%e33 #&#
\begin{equation}
\label{martn} \qquad M^{(0)}(dx) = dx,\qquad M^{(k)}(dx) =
\lim_{\varepsilon
\searrow0} e^{X^{(k),\varepsilon}_{x} - (\psi(1)/n) \ln
(1/\varepsilon)} M^{(k-1)}(dx),
\end{equation}
where the limit has to be understood in the sense of weak convergence
of Radon measures. For $k \in[1,n-1]$, one has the relation
\[
\frac{\widebar{F}}{n} \leq d - \frac{k}{n}\bigl(\psi'(1) -
\psi(1)\bigr),
\]
so that we can apply recursively Lemma \ref{Int} to prove that for each
$k \leq n$,
\[
\mathbb E \bigl[M^{(k)}(B)\bigr] = \mathbb E \bigl[M^{(k-1)}(B)
\bigr]\quad\mbox{and}\quad M^{(k)} \in R_{d-(k/n)(\psi'(1) - \psi(1))}.
\]
In particular, the martingales considered in (\ref{martn}) are
uniformly integrable. Then we prove that the measures $M$ and $M^{(n)}$
have the same law. For this, we note that the following equality in law
holds:
%e33 ###
%e34 #&#
\begin{equation}
\label{lawm} M^{(n)}(dx) = \lim_{\varepsilon\searrow0 }
e^{X^{(1),\varepsilon
}_{x} + \cdots+ X^{(n),\varepsilon}_{x} - \psi(1) \ln(1/\varepsilon
)} \,dx.
\end{equation}
Indeed, consider the $\sigma$-algebra $\mathcal{G}_{\varepsilon}$
generated by $\{ X_{r}^{(1),\varepsilon'},
\ldots,X_{r}^{(n),\varepsilon'}, \varepsilon' > \varepsilon, r
\in\mathbb{R}^d \} $. Using the fact that the martingales considered in
(\ref{martn}) are uniformly integrable, we compute
\begin{eqnarray*}
&& \mathbb E \bigl[ M^{(n)}(A) | \mathcal{G}_{\varepsilon} \bigr]
\\
&&\qquad = \mathbb E \bigl[ \mathbb E \bigl[ M^{(n)}(A) \bigl| \bigl
(X_{r}^{(1),\varepsilon'},
\ldots,X_{r}^{(n-1),\varepsilon'}\bigr)_{r \in
\mathbb{R}^d,\epsilon'\in]0,1[},
\bigl(X^{(n),\varepsilon'}_{r}\bigr)_{r \in
\mathbb{R}^d, \varepsilon' > \varepsilon} \bigr] \bigr|
\mathcal{G}_{\varepsilon} \bigr]
\\
&&\qquad  = \mathbb E \bigl[ \mathbb E ^{(n)} \bigl[ M^{(n)}(A) |
\bigl(X^{(n),\varepsilon'}_{r}\bigr)_{r \in\mathbb{R}^d, \varepsilon' >
\varepsilon} \bigr] |
\mathcal{G}_{\varepsilon} \bigr]
\\
&&\qquad = \mathbb E \biggl[ \int_{A} e^{X^{(n),\varepsilon}_{r} -
(\psi(1)/n)\log(1/\varepsilon)}
M^{(n-1)}(dr) \Big| \mathcal{G}_{\varepsilon} \biggr]
\\
&&\qquad = \cdots
\\
&&\qquad = \int_{A} e^{X^{(n),\varepsilon}_{r} + \cdots+ X^{(1),\varepsilon
}_{r} - \psi(1)\log(1/\varepsilon)} \,dr.
\end{eqnarray*}
Since this last quantity has the same law as $M^\epsilon(A)$, (\ref
{lawm}) follows by passing to the limit as $\epsilon\to0$. Since $
\mathbb E [M^{(n)}(A)]=|A|$, we deduce $ \mathbb E [M(A)]=|A|$. Hence
$M$ is not trivial. Furthermore we have proved that $M\in R_{d-\psi
'(1)+\psi(1)}$. In particular, $M$~cannot possess any atom.
\end{pf*}

We further stress that the proof of Theorem \ref{AMontkappa} works
exactly the same [just replace $dx$ by $\kappa(dx)$ in the proof of
Theorem \ref{a_mont}].
%%%%%%%%%%%%%%%%%%%%%%%%%%%%%%%%%%%%%%%%%%%%%%%%%%%%%%%%%%%%%%%%%%%%%%%%%%%%%%%%%%%%%
%s8 ###
%s8 #&#
\section{\texorpdfstring{Proofs of Section \lowercase{\protect\ref{secstar}}}{Proofs of Section 3}}
%%%%%%%%%%%%%%%%%%%%%%%%%%%%%%%%%%%%%%
%%%%%%%%%%%%%%%%%%%%%%%%%%%%%%%%%%%%%%

%%%%%%%%%%%%%%%%%%%%%%%%%%%%%%%%%%%%%%
%s8.1 ###
%s8.1 #&#
\subsection{\texorpdfstring{Proof of Proposition \protect\ref{Fstar}}{Proof of Proposition 11}}
We have
\begin{eqnarray*}
\int_a^{\infty}\frac{|F(ux)|}{u} \,du&=&\int
_{a|x|}^{\infty}\frac
{|F(ue_x)|}{u} \,du,
\end{eqnarray*}
where $e_x=\frac{x}{|x|}$. For $a|x|\geq1$, this quantity is less than
(\ref{sphere}). For $a|x|\leq1$, we have the bound
\begin{eqnarray*}
\int_{a|x|}^{\infty}\frac{|F(ue_x)|}{u} \,du & =& \int
_{a|x|}^{1}\frac{|F(ue_x)|}{u} \,du+\int
_{1}^{\infty}\frac
{F(ue_x)}{u} \,du
\\[-1pt]
&\leq& F(0)\ln\frac{1}{a|x|}+\int_{1}^{\infty}
\frac{|F(ue_x)|}{u} \,du
\end{eqnarray*}
because $|F(x)|\leq F(0)$. Actually, because of the continuity of the
function $F$ at $0$, it turns out that we have $\int_{|x|}^{1}\frac
{F(ue_x)}{u} \,du\simeq F(0)\ln\frac{1}{|x|}$ as $|x|\to0$. We deduce
%e34 ###
%e35 #&#
\begin{equation}
\label{equiv} \int_{ |x|}^{\infty}\frac{F(ue_x)}{u} \,du
\simeq F(0)\ln\frac
{1}{|x|}\qquad\mbox{as }|x|\to0.
\end{equation}

%%%%%%%%%%%%%%%%%%%%%%%%%%%%%%%%%%%%%%
%s8.2 ###
%s8.2 #&#
\subsection{\texorpdfstring{Proof of Proposition \protect\ref{mom2}}{Proof of Proposition 13}}
We just have
to compute the second order moment (we use the notation $e_{x-y}=\frac
{x-y}{|x-y|}$)
\begin{eqnarray*}
\mathbb E \bigl[\widetilde{M}{}^{\epsilon}(A)^{2}\bigr] &=& \int
_{A \times A} \mathbb E \bigl[e^{X_{x}^{\epsilon}+ X_{y}^{\epsilon}}
\bigr]
e^{-2\psi(1)\ln(1/\epsilon)} \,dx\,dy
\\[-1pt]
&=& \int_{A \times A} e^{\int_{1}^{1/\varepsilon}F (u(x-y)
)((du)/u)} \,dx\,dy
\\[-1pt]
& =&\int_{A \times A} e^{\int_{|x-y|}^{|x-y|/\varepsilon}F (u
e_{x-y} )((du)/u)} \,dx\,dy.
\end{eqnarray*}
In case $M$ admits a second order moment, we deduce that the quantity
\[
\mathbb E \bigl[M(A)^{2}\bigr]=\int_{A \times A}
e^{\int_{|x-y|}^{\infty}F
(u e_{x-y} )((du)/u)} \,dx\,dy
\]
is finite. Because of (\ref{equiv}), we necessarily have $F(0)<d$.
Conversely, if $F(0)<d$, then $\sup_\epsilon\mathbb E [\widetilde
{M}{}^{\epsilon}(A)^{2}]$ is less than the above right-hand side, which
is finite. The proof is complete.

%s8.3 ###
%s8.3 #&#
\subsection{\texorpdfstring{Proof of Proposition \protect\ref{propscaleinv}}{Proof of Proposition 12}}
For $0<\epsilon<1$, $t_1,\ldots,t_p\in(\mathbb{R}^d)^p$ and
$q_1,\ldots,\break  q_p\in\mathbb{R}$ such that the following expectations make
sense, we define the Laplace exponents $ \psi^\epsilon$ of $X^\epsilon$
\[
\mathbb E \bigl[e^{q_1X^\epsilon_{t_1}+\cdots+q_pX^\epsilon
_{t_p}}\bigr]=e^{\psi^\epsilon_{t_1,\ldots,t_p}(q_1,\ldots,q_p)}.
\]
For $\epsilon'<\epsilon$, we have
\begin{eqnarray*}
&& \psi^{\epsilon'}_{t_1,\ldots,t_p}(q_1,\ldots,q_p)
\\[-5pt]
&&\quad =b\ln\frac{1}{\epsilon'}\sum_{i=1}^pq_i
+ \frac{1}{2}\int_1^{1/\epsilon'}\!\!\!\int
_{\mathbb{R}^d} \Biggl(\sum_{i=1}^pq_i
\cos(y t_iu) \Biggr)^2 R(du)\frac{dy}{y}
\\[-1pt]
&&\qquad{} + \frac{1}{2}\int_1^{1/\epsilon'}\!\!\!\int
_{\mathbb{R}^d} \Biggl(\sum_{i=1}^pq_i
\sin(y t_iu) \Biggr)^2 R(du)\frac{dy}{y}
\\
&&\qquad{}+ \int_1^{1/\epsilon'}\!\!\!\int_S
\Biggl(e^{\sum_{i=1}^pq_if(T_{t_i y}(s))}-1-\sum_{i=1}^pq_i
\frac{f(T_{t_i y}(s))}{1\vee|f(T_{t_i y}(s))|} \Biggr)\theta(ds)\frac{dy}{y}
\\
&&\quad =b\ln\frac{\epsilon}{\epsilon'}\sum_{i=1}^pq_i
+ \frac{1}{2}\int_{1/\epsilon}^{1/\epsilon'}\!\!\!\int
_{\mathbb{R}^d} \Biggl(\sum_{i=1}^pq_i
\cos(y t_iu) \Biggr)^2 R(du)\frac
{dy}{y}
\\
&&\qquad{}+ \frac{1}{2}\int_{1/\epsilon}^{1/\epsilon'}\!\!\!\int
_{\mathbb{R}^d} \Biggl(\sum_{i=1}^pq_i
\sin(y t_iu) \Biggr)^2 R(du)\frac{dy}{y}
\\
&&\qquad{}+ \int_{1/\epsilon}^{1/\epsilon'}\!\!\!\int_S
\Biggl(e^{\sum_{i=1}^pq_if(T_{t_i y}(s))}-1-\sum_{i=1}^pq_i
\frac{f(T_{t_i
y}(s))}{1\vee|f(T_{t_i y}(s))|} \Biggr)\theta(ds)\frac{dy}{y}
\\
&&\qquad{} +\psi ^{\epsilon}_{t_1,\ldots,t_p}(q_1,\ldots,q_p)
\\
&&\quad =b\ln\frac{\epsilon}{\epsilon'}\sum_{i=1}^pq_i
+ \frac{1}{2}\int_{1}^{\epsilon/\epsilon'}\!\!\!\int
_{\mathbb{R}^d} \Biggl(\sum_{i=1}^pq_i
\cos\biggl(y \frac{t_i}{\epsilon}u\biggr) \Biggr)^2 R(du)
\frac{dy}{y}
\\
&&\qquad{} +\frac{1}{2}\int_{1/\epsilon}^{1/\epsilon'}\!\!\!\int_{\mathbb{R}^d} \Biggl(\sum_{i=1}^pq_i
\sin\biggl(y \frac
{t_i}{\epsilon}u\biggr) \Biggr)^2 R(du)
\frac{dy}{y}
\\
&&\qquad{}+ \int_{1}^{\epsilon/\epsilon'}\!\!\!\int_S
\Biggl(e^{\sum_{i=1}^pq_if(T_{(t_i/\epsilon) y }(s))}-1-\sum_{i=1}^pq_i
\frac{f(T_{(t_i/\epsilon) y}(s))}{1\vee|f(T_{(t_i/\epsilon)y}(s))|} \Biggr)\theta(ds)\frac{dy}{y}
\\
&&\qquad{}+\psi^{\epsilon}_{t_1,\ldots,t_p}(q_1,\ldots,q_p)
\\
&&\quad =\psi^{\epsilon'/\epsilon}_{t_1/\epsilon,\ldots,t_p/\epsilon}(q_1,\ldots,q_p)+
\psi^{\epsilon}_{t_1,\ldots,t_p}(q_1,\ldots,q_p).
\end{eqnarray*}

Hence we can write
%e35 ###
%e36 #&#
\begin{equation}
\label{eqsum} \bigl(X^{\epsilon'}_x\bigr)_x
\stackrel{\mathrm{law}} {=}\bigl(X^{\epsilon
}_x+
\widebar{X}{}^{\epsilon'/\epsilon}_{x/\epsilon}\bigr)_x,
\end{equation}
where $\widebar{X}{}^{\epsilon'/\epsilon}$ is independent from
$X^{\epsilon}$ and has the same law as
$X^{\epsilon'/\epsilon}$. It is then plain to deduce that $M$
is $\star$-scale invariant. Indeed, define $M^\epsilon$ by
\[
\forall A\in\mathcal{B}\bigl(\mathbb{R}^d\bigr)\qquad
M^\epsilon(A)=\lim_{\epsilon'\to0}\int_Ae^{\widebar{X}{}^{\epsilon'/\epsilon}_{x/\epsilon}-\psi(1)\ln(\epsilon/\epsilon')}
\,dx.
\]
A straightforward change of variables shows that
\[
M^\epsilon(dx)\stackrel{\mathrm{law}} {=}\epsilon^dM(dx/
\epsilon).
\]
From (\ref{eqsum}), we deduce
\[
M(dx)=e^{X^{\epsilon}_x-\psi(1)\ln(1/\epsilon)}M^\epsilon(dx).
\]

%%%%%%%%%%%%%%%%%%%%%%%%%%%%%%%%%%%%%%%%%%
%%%%%%%%%%%%%%%%%%%%%%%%%%%%%%%%%%%%%%%%%%

%s9 ###
%s9 #&#
\section{\texorpdfstring{Proof of Theorem \lowercase{\protect\ref{thuniq}}}{Proof of Theorem 16}}
%%%%%%%%%%%%%%%%%%%%%%%%%%%%%%%%%%%%%%%%%%
%%%%%%%%%%%%%%%%%%%%%%%%%%%%%%%%%%%%%%%%%%
We carry out the proof in the case when the dimension is equal to $1$.
This simplifies the notation. In higher dimensions, the proof works the
same way.

The guiding line is the same as in \cite{allez}. But the lack of
convexity inequalities, which are specific to the Gaussian case, gives
rise to further technical difficulties. So we detail what differs and
refer to \cite{allez} for the proofs of the results that do not change
with respect to the Gaussian case.

%s9.1 ###
%s9.1 #&#
\subsection{Setting}
%%%%%%%%%%%%%%%%%%%%%%%%%%%%%%%%%%%%%%
We consider a nontrivial measure satisfying (\ref{star}) with a moment
of order $1+\delta$ for some $\delta>0$ and a fixed $\epsilon\in\,]0,1[$.
The first step is to prove that the measure $M$ is a L\'evy
multiplicative chaos. Since $M$ is not trivial and possesses a moment
of order at least $1$, we necessarily have
%e36 ###
%e37 #&#
\begin{equation}
\label{norm} \forall x\in\mathbb{R}\qquad\mathbb E \bigl[e^{\omega
_\epsilon(x)}
\bigr]=1.
\end{equation}
Because it is stochastically continuous and ID, the process
$\omega_\epsilon$ admits a version with a representation as in
(\ref{version}) with associated parameters
$(S_\epsilon,W_\epsilon,W'_\epsilon,\break N,\theta_{\varepsilon
},R_{\varepsilon}, f_{\varepsilon},(T^\epsilon_x)_x)$. The Laplace
transform of $ \omega_\epsilon$ is denoted by
\[
\psi_\epsilon(q)=\ln\mathbb E \bigl[e^{q\omega_\epsilon(0)}\bigr].
\]
It satisfies $ \psi_{\varepsilon}(1) =0$. We let $(X^{n})_n$ denote a
sequence of independent stationary stochastically continuous ID
processes with common law that of $\omega_{\varepsilon}$. Of course,
the law
of this sequence depends on $\epsilon$, but we remove this dependence
from the notation for the sake of clarity. We also define the measure
$M^{N}$ for $N \geq0$ by
%e37 ###
%e38 #&#
\begin{equation}
M^{N}(A) = \varepsilon^{N+1}M \biggl( \frac{1}{\varepsilon
^{N+1}}A
\biggr).
\end{equation}
We assume that the sequences $(X_n)_n$ and $(M^N)_N$ are independent.
Iterating relation (\ref{star}), we get that, for every integer $N$,
the measure $\widetilde{M}{}^{N}$ defined by
%e38 ###
%e39 #&#
\begin{equation}
\label{itt} \widetilde{M}{}^{N}(A) = \int_{A}
\exp\Biggl( \sum_{n=0}^{N}
X^{n}_{r/\varepsilon^{n}} \Biggr) M^{N}(dr)
\end{equation}
has the same law as the measure $M$.

%le27 #&#
\begin{lemma}[(See \cite{allez})]\label{birk}
Let $M$ be a stationary random measure on $\mathbb{R}$ admitting a
moment of
order $1+\delta$. There is a nonnegative integrable random variable
$Y\in L^{1+\delta}$ such that, for every bounded interval $I\subset
\mathbb{R}$,
\[
\lim_{T \to\infty} \frac{1}{T} M (T I ) = Y |I|\qquad
\mbox{almost surely and in }L^{1+\delta},
\]
where $|\cdot|$
stands for the Lebesgue measure on $\mathbb{R}$. As a consequence, almost
surely the random measure
\[
A\in\mathcal{B}(\mathbb{R})\mapsto\frac{1}{T}M(TA)
\]
weakly converges toward $Y|\cdot|$, and
$ \mathbb E _Y[M(A)]=Y |A|$ ($ \mathbb E _Y[\cdot]$ denotes the
conditional expectation
with respect to $Y$).
\end{lemma}

Thus, in what follows, the random variable $Y$ will be defined as the
unique (up to a set of probability $0$) random variable such that
$ \mathbb E _Y[M(A)]=Y |A|$ for all Borel sets $A$.

For $x\neq0$, define
%e39 ###
%e40 #&#
\begin{equation}
\label{gencov} S^{\varepsilon}(x) = \sum_{n=0}^{\infty}
F_{\varepsilon
}\bigl(x/\varepsilon^{n}\bigr),
\end{equation}
where $F_{\varepsilon}(\cdot)$ is the generalized covariance function
associated with $\omega_{\varepsilon}$; see Assumption~\ref{good}. The
uniform convergence of the series on the sets of the type $\{x\in
\mathbb{R};|x|\geq\rho\}$ is ensured by (\ref{modulus}); see \cite
{allez}. Then we can reproduce the proofs of \cite{allez}, Section~5.2,
by replacing $K^\epsilon$ by $S^{\varepsilon}$ in the proofs.%and the
%limit $S^\epsilon$ is locally Lipschitz on $\R\setminus\{0\}$

% As a consequence \cite{allez}, we obtain:
%For any Lebesgue integrable function $\phi$ on $\R^2$ and $d>0$, we
%have for all $N\in\N\setminus\{0\}$:
%for some function $\xi:\R_+\to\R_+$ such that $\lim_{d\to\infty}

%s9.2 ###
%s9.2 #&#
\subsection{$M$ is a L\'evy multiplicative chaos}

%%%%%%%%%%%%%%%%%%%%%%%%%%%%%%%%%%%
Let us define the $\sigma$ algebra $\mathcal{F}_{N} =
\sigma(X^{0},\ldots, X^{N},Y)$. For every Borel subset $A
\subset\mathbb{R}$, we define
%
%e40 ###
%e41 #&#
\begin{equation}
\label{gnexp} G_{N}(A) = \mathbb E \bigl[ \widetilde{M}{}^{N}(A)|
\mathcal{F}_{N} \bigr].
\end{equation}
As in \cite{allez}, we prove
%e41 ###
%e42 #&#
\begin{equation}
\forall N \geq0\qquad G_{N}(A) = Y \int_{A}
\exp\Biggl( \sum_{n=0}^{N}
X^{n}_{x/\varepsilon^{n}} \Biggr) \,dx.
\end{equation}
Hence,\vspace*{1pt} for each bounded Borel set $A$, the sequence
$ (G_{N}(A) )_{N}$ is a positive martingale
bounded in $L^{1+\delta}$. Being bounded in $L^{1+\delta}$, the
martingale $G_{N}(A)$ converges toward a random variable $Q(A)$ which
should be formally thought of as
\[
Q(A) = Y \int_{A} \exp\Biggl( \sum
_{n=0}^{\infty} X^{n}_{x/\varepsilon^{n}} \Biggr)
\,dx.
\]

% \end{pf}
%
% \subsection{Key lemma}
%
% The aim of this section is to prove the following key result:

The result below is proved in \cite{allez} and uses specific properties
of Gaussian processes, namely Gaussian concentration inequalities due
to Kahane; see \cite{cfKah}. It turns out that we can carry out the
proof while skipping these inequalities:
%
%le28 #&#
\begin{lemma}\label{Gieq}
For small enough $ \gamma\in\,]0,\delta[$, there exists $\rho> 0$
such that
%e42 ###
%e43 #&#
\begin{equation}
\sup_{n} n^{1 + \rho} \mathbb E \biggl[ M \biggl(
\biggl[0,\frac
{1}{n} \biggr] \biggr)^{1+\gamma} \biggr] < \infty.\vadjust{\goodbreak}
\end{equation}
\end{lemma}

The central lemma for establishing Lemma \ref{Gieq} is the following:

%le29 #&#
\begin{lemma}\label{lemfatal}
The finiteness of a moment of order $1+\delta$ (\,for some $\delta>0$)
implies
%e43 ###
%e44 #&#
\begin{equation}
\label{fatalkitu} \forall\epsilon<1\qquad\psi_\epsilon'(1)<\ln
\frac{1}{\epsilon}
\end{equation}
and
%e44 ###
%e45 #&#
\begin{equation}
\label{fatalkitu2} \forall\gamma\in[0,\delta[,\ \forall\epsilon<1\qquad
\psi
_\epsilon(1+\gamma) <\gamma\ln\frac{1}{\epsilon}.
\end{equation}
\end{lemma}

\begin{pf}
Let us fix $\epsilon<1$ and define for $q\leq1$,
\[
F_\epsilon(q,r)=\ln\mathbb E \bigl[e^{q\omega_\epsilon(r)+q\omega
_\epsilon(0)}\bigr].
\]
Let us consider $h>1$ such that $(1+\delta)h=2$. By concavity of the
function $x \mapsto x^{1/h}$, we can make use of Jensen's inequality to
get for $N\geq1$,
\begin{eqnarray*}
&& \mathbb E \biggl[ M \biggl(\biggl[0,\frac{1}{n}\biggr]
\biggr)^{1+\delta} \biggr]
\\
&&\qquad  = \mathbb E \biggl[ M \biggl(\biggl[0,
\frac{1}{n}\biggr] \biggr)^{2\times(1/h)} \biggr]
\\
&&\qquad = \mathbb E \biggl[ \biggl( \int_{0}^{1/n} \int
_{0}^{1/n}e^{\sum
_{p=0}^{N-1}X^p(r/\epsilon^p)+X^p(u/\epsilon^p)} M^{N-1}(dr)
M^{N-1}(du) \biggr)^{1/h} \biggr]
\\
&&\qquad \geq \mathbb E \biggl[ \int_{0}^{1/n} \int
_{0}^{1/n}e^{(1/h)\sum_{p=0}^{N-1}X^p(r/\epsilon^p)+X^p(u/\epsilon^p)} M^{N-1}(dr)
M^{N-1}(du) M^{N-1}
\\
&&\hspace*{267pt}{}\times \biggl(\biggl[0,\frac
{1}{n}\biggr]\biggr)^{2/h-2} \biggr]
\\
&&\qquad \geq e^{N\inf_{|r|\leq 1/(n\epsilon^N)} F_\epsilon(1/h,r)} \mathbb E \biggl[ M^{N-1}\biggl(\biggl[0,
\frac{1}{n}\biggr]\biggr)^{1+\delta} \biggr],
\end{eqnarray*}
where we made use of the fact that the sequence $(X^{n})_{n\leq N-1}$
is independent of the random measure $M^{N-1}$. Now we choose $N$ such
that $\epsilon^N=\frac{1}{\alpha n}$ for some $\alpha>0$, that is,
$N=\frac{\ln\alpha+\ln n}{\ln(1/\epsilon)} $. We obtain
\begin{eqnarray*}
&& \mathbb E \biggl[ M \biggl( \biggl[0,\frac{1}{n} \biggr]
\biggr)^{1+\delta
} \biggr]
\\
&&\qquad \geq\frac{1}{n^{1+\delta}}e^{((\ln n +\ln\alpha)
(\inf_{|r|\leq\alpha} F_\epsilon(1/h,r))/\ln(1/\epsilon))}
\frac{1}{\alpha^{1+\delta}} \mathbb E \bigl[ M \bigl([0,\alpha]\bigr
)^{1+\delta} \bigr].
\end{eqnarray*}
Now, we use the super-additivity of the function $x\mapsto
x^{1+\delta}$ to obtain
\begin{eqnarray*}
\mathbb E \bigl[ M \bigl( [0,1 ] \bigr)^{1+\delta} \bigr] &\geq&\sum
_{k=1}^n \mathbb E \biggl[ M \biggl(
\biggl[\frac{k-1}{n},\frac
{k}{n} \biggr] \biggr)^{1+\delta}
\biggr].
\end{eqnarray*}
By gathering the above inequalities, we deduce
\begin{eqnarray*}
&& \mathbb E \bigl[ M \bigl( [0,1 ] \bigr)^{1+\delta} \bigr]
\\
&&\qquad \geq n
\frac{1}{n^{1+\delta}}e^{(\ln n +\ln\alpha) ((\inf_{|r|\leq
\alpha} F_\epsilon(1/h,r))/\ln(1/\epsilon))}\frac
{1}{\alpha^{1+\delta}} \mathbb E \bigl[ M \bigl([0,
\alpha]\bigr)^{1+\delta} \bigr].
\end{eqnarray*}
Because the left-hand side is bounded independently of $n$, we
necessarily have
%e45 ###
%e46 #&#
\begin{equation}
\forall\epsilon>0\qquad\frac{\inf_{|r|\leq\alpha} F_\epsilon
(1/h,r)}{\ln(1/\epsilon)}\leq\delta.
\end{equation}
By letting $\alpha$ go to $0$ and by continuity of
$F_\epsilon(\frac{1}{h},\cdot)$ at $0$ ($\omega_\epsilon$ is
stochastically continuous with a moment of order $1+\delta$), we
deduce
%e46 ###
%e47 #&#
\begin{equation}
\label{semifatal} \forall\epsilon>0\qquad\frac{ \psi_\epsilon(1+\delta
)}{\ln(1/\epsilon)}\leq\delta.
\end{equation}
By convexity arguments, it is then plain to deduce that
%e47 ###
%e48 #&#
\begin{equation}
\label{fatal} \frac{ \psi_\epsilon'(1)}{\ln(1/\epsilon)}<1.
\end{equation}
Indeed, the (not strict) inequality results from (\ref{semifatal}). If
equality holds, this means that $\psi_\epsilon(1+\gamma)=1$ for all
$\gamma\in[0,\delta[$. By analycity arguments, this implies that the
law of the process $\omega_\epsilon$ is that of a constant, and the
measure $M$ is thus trivial. This is in contradiction to our
assumptions. The same type of argument leads to (\ref{fatalkitu2}).
\end{pf}

\begin{pf*}{Proof of Lemma \ref{Gieq}}
We consider $\gamma\in\,]0,\delta[$. As the function $x \mapsto x^{1+\gamma}$ is convex, we
make use of Jensen's inequality to get for $N\geq1$,
\begin{eqnarray*}
&& \mathbb E \biggl[ M \biggl( \biggl[0,\frac{1}{n} \biggr]\biggr)^{1+\gamma
} \biggr]
\\
&&\qquad  = \mathbb E \biggl[ \biggl( \int_{0}^{1/n} \frac{e^{\sum
_{p=0}^{N-1}X^p(r/\epsilon^p)}}{M^{N-1}([0,1/n])} M^{N-1}(dr)
\biggr)^{1+\gamma} M^{N-1} \biggl( \biggl[0,\frac{1}{n}
\biggr] \biggr)^{1+\gamma} \biggr]
\\
&&\qquad \leq \mathbb E \biggl[ \int_{0}^{1/n}e^{(1+\gamma)\sum
_{p=0}^{N-1}X^p(r/\epsilon^p)} M^{N-1}(dr) M^{N-1} \biggl(
\biggl[0,\frac{1}{n} \biggr] \biggr)^{\gamma} \biggr]
\\
&&\qquad \leq \mathbb E \bigl[ e^{(1+\gamma)\sum_{p=0}^{N-1}X^p(0)} \bigr]
\mathbb E \biggl[
M^{N-1} \biggl( \biggl[0,\frac{1}{n} \biggr]
\biggr)^{1+\gamma} \biggr]
\\
&&\qquad = e^{N\psi_\epsilon(1+\gamma)} \mathbb E \biggl[ M^{N-1} \biggl( \biggl[0,
\frac{1}{n} \biggr] \biggr)^{1+\gamma} \biggr],
\end{eqnarray*}
where, once again, we made use of the fact that the sequence
$(X^{n})_{n\leq N-1}$ is independent of the random measure $M^{N-
1}$. We choose $N=\frac{-\ln n}{\ln\epsilon}$ in order to have
$\epsilon^N=\frac{1}{n}$. We get that
\begin{eqnarray*}
\mathbb E \biggl[ M \biggl( \biggl[0,\frac{1}{n} \biggr]
\biggr)^{1+\gamma
} \biggr] & \leq&\frac{1}{n^{1+\gamma-((\psi_\epsilon(1+\gamma))/\ln(1/\epsilon))}} \mathbb E \bigl[ M \bigl(
[0,1 ] \bigr)^{1+\delta} \bigr].
\end{eqnarray*}
We are thus left with checking that $\frac{\psi_\epsilon'(1)}{\ln(1/\epsilon)} < 1$. This
the content of Lemma~\ref{lemfatal}.
\end{pf*}

Let us stress that, as an immediate consequence of Lemma \ref{Gieq},
the measure $M$ does not possess any atom; see \cite{daley},
Corollary~9.3VI. With the above estimation on
the function $\psi_\epsilon$, we can prove that $Q$ is a nontrivial
L\'evy multiplicative chaos:
%
%le30 #&#
\begin{lemma}\label{trivQ}
The random measure $Q$ is a L\'evy multiplicative chaos, and it is
nonrivial.
\end{lemma}

\begin{pf} Let us use the decomposition of $X^n$ to write
\begin{eqnarray*}
X^n(r)&=& b_\epsilon+\int_{\mathbb{R}^d}\cos(t\cdot
u) W_\epsilon^n(du)+\int_{\mathbb{R}^d}\sin(t\cdot
u) W_\epsilon^{n'}(du)
\\
&&{} +\int_{S_\epsilon}f_\epsilon
\bigl(T^\epsilon_t(s)\bigr) \bigl[N_\epsilon^n(ds)-
\bigl(1\vee\bigl|f_\epsilon\bigl(T^\epsilon_t(s)\bigr)\bigr|
\bigr)^{-1}\theta_\epsilon(ds)\bigr],
\end{eqnarray*}
where the triples $(W_\epsilon^n,W_\epsilon^{n'},N_\epsilon^n)_n$
are independent. Thus we have
\begin{eqnarray*}
Y^N_t&=&\sum_{n=0}^NX^n
\biggl(\frac{t}{\epsilon^n}\biggr)
\\
&=&Nb_\epsilon+\sum_{n=0}^N\int
_{\mathbb{R}^d}\cos\biggl(\frac
{t}{\epsilon^n}\cdot u\biggr)
W_\epsilon^n(du)+\sum_{n=0}^N
\int_{\mathbb
{R}^d}\sin\biggl(\frac{t}{\epsilon^n}\cdot u\biggr)
W_\epsilon^{n'}(du)
\\
&&{}+\sum_{n=0}^N\int_{S_\epsilon}f_\epsilon
\bigl(T^\epsilon_{t/\epsilon^n}(s)\bigr)
\bigl[N_\epsilon^n(ds)-\bigl(1\vee\bigl|f_\epsilon
\bigl(T^\epsilon_{t/\epsilon^n}(s)\bigr)\bigr|
\bigr)^{-1}\theta_\epsilon(ds)\bigr].
\end{eqnarray*}

Let us compute the L\'evy exponent of $Y^N$. For $r_1,\ldots,r_p\in
\mathbb{R}$ and $\lambda_1,\ldots,\break  \lambda_N\in\mathbb{R}$ such that the
following expectations make sense, we have
\begin{eqnarray*}
\hspace*{-5pt}&& \mathbb E \bigl[e^{\sum_{i=1}^p\lambda_iY^N_{r_i}}\bigr]
\\
\hspace*{-5pt}&&\quad = \exp\Biggl(Nb_\epsilon+
\frac{1}{2}\sum_{n=0}^N\int
_{\mathbb{R}^d} \Biggl(\sum_{i=1}^p
\lambda_i\cos\biggl(\frac{r_i}{\epsilon^n}\cdot u\biggr)
\Biggr)^2R_\epsilon(du)
\\
\hspace*{-5pt}&&\hspace*{21pt}\qquad{}+\frac{1}{2}\sum
_{n=0}^N\int_{\mathbb{R}^d} \Biggl(\sum
_{i=1}^p\lambda_i\sin\biggl(
\frac{r_i}{\epsilon^n}\cdot u\biggr) \Biggr)^2R_\epsilon(du)
\\
\hspace*{-5pt}&&\hspace*{21pt}\qquad{}+\sum_{n=0}^N\int_{S_\epsilon}
\Biggl(e^{\sum_{i=1}^p\lambda_i
f_\epsilon(T^\epsilon_{r_i/\epsilon^n}(s))}-1-\sum_{i=1}^p
\frac{\lambda_if_\epsilon(T^\epsilon_{r_i/\epsilon
^n}(s))}{1\vee|f_\epsilon(T^\epsilon_{r_i/\epsilon
^n}(s))|} \Biggr)\theta_\epsilon(ds) \Biggr).
\end{eqnarray*}
We point out that the last quantity can be rewritten as
\begin{eqnarray*}
\hspace*{-6pt}&& \exp\Biggl(\int_1^{1/\epsilon^N}\frac{b_\epsilon}{\ln(1/\epsilon)}
\frac{dy}{y}+ \frac{1}{2}\int_1^{1/\epsilon^N}
\!\!\!\int_{\mathbb{R}^d} \Biggl(\sum_{i=1}^p
\lambda_i\cos\bigl(r_ig(y)\cdot u\bigr)
\Biggr)^2\frac{R_\epsilon(du)}{\ln(1/\epsilon)}\frac{dy}{y}
\\
\hspace*{-6pt}&&\hspace*{12pt}{}+ \frac{1}{2}\int_1^{1/\epsilon^N}\!\!\! \int
_{\mathbb
{R}^d} \Biggl(\sum_{i=1}^p
\lambda_i\sin\bigl(r_ig(y)\cdot u\bigr)
\Biggr)^2\frac
{R_\epsilon(du)}{\ln(1/\epsilon)}\frac{dy}{y}
\\
\hspace*{-6pt}&&\hspace*{12pt}{}+\int_1^{1/\epsilon^N}\!\!\!\int_{S_\epsilon}
\Biggl(e^{\sum
_{i=1}^p\lambda_i f_\epsilon(T^\epsilon_{r_ig(y)} (s))}-1-\sum_{i=1}^p
\frac{\lambda_if_\epsilon(T^\epsilon_{r_ig(y)}(s))}{1\vee
|f_\epsilon(T^\epsilon_{r_ig(y)}(s))|} \Biggr)\frac{\theta_\epsilon
(ds)}{\ln(1/\epsilon)}\frac{dy}{y} \Biggr),
\end{eqnarray*}
where $g$ is defined by $g(y)=\frac{1}{\epsilon^n}$ on the interval
$[\frac{1}{\epsilon^n},\frac{1}{\epsilon^{n+1}}[$. Hence, $Q$ is
obviously a L\'evy multiplicative chaos. Furthermore, from relation
(\ref{gnexp}), it is plain to deduce that the martingale $(G_N(A))_N$
is bounded in $L^{1+\delta}$ as $(\widetilde{M}{}^N)_N$ is. Thus, the
martingale $(G_N(A))_N$ converges a.s. and in $L^{1+\delta}$ toward its
limit $Q(A)$, which is necessarily nontrivial.
\end{pf}

Once we have proved Lemmas \ref{Gieq} and \ref{trivQ}, we can proceed
along the same lines as in \cite{allez}, Section~5, to have the
following description of the set of good $\star$-scale invariant random
measures:
%
%pr31 #&#
\begin{proposition}\label{Mchao}
$\!\!\!$The random measures $(Q(A))_{A \in\mathcal{B}(\mathbb{R})}$ and $(M(A))_{A
\in\mathcal{B}(\mathbb{R})}$ have the same law.
\end{proposition}

%%%%%%%%%%%%%%%%%%%%%%%%%%%%%%%%%%%%%%%%
%s9.3 ###
%s9.3 #&#
\subsection{Structure of the L\'evy chaos}
%%%%%%%%%%%%%%%%%%%%%%%%%%%%%%%%%%%%%%%%
Now, we still have to show that the chaos $M$ can be recovered in the
same way as the construction set out in Section~\ref{secstar}. For
$(t_1,\ldots,t_p)\in\mathbb{R}^p$, we introduce the L\'evy exponent
$\eta^\epsilon$ of the random variable $(\omega
_{\epsilon}(t_1),\ldots,\omega_{\epsilon}( t_p))$, namely
\[
\mathbb E \bigl[e^{iq_1\omega_{\epsilon}( t_1)+\cdots+iq_p\omega
_{\epsilon}( t_p)}\bigr]=e^{\eta^\epsilon(t_1,\ldots,t_p,Q) },
\]
where\vspace*{-1pt} $Q$ stands for the vector $(q_1,\ldots,q_p)\in\mathbb{R}^p$. For
$T\in\mathbb{R}$, $Q\in\mathbb{R}^p$ and $t\geq0$, we define $G_Q
(t,T)=\eta^{e^{-t}}(e^{-T}t_1,\ldots,e^{-T}t_p,Q)$. It is the L\' evy
exponent of the random variable
$(\omega_{e^{-t}}(e^{-T}t_1),\ldots,\omega_{e^{-t}}(e^{-T}t_p))$. Now,
we make use of the cascading equation. We claim that for $\varepsilon,
\varepsilon' > 0$, the following equality holds:
\[
\bigl(\omega_{r}^{\varepsilon\varepsilon'}\bigr)_{r\in\mathbb
{R}}\stackrel{
\mathrm{law}} {=} \bigl(\omega_{r}^{\varepsilon} +
\omega_{r/\varepsilon
}^{\varepsilon'}\bigr)_{r\in\mathbb{R}},
\]
where the processes $\omega^{\varepsilon}$ and $\omega^{ \varepsilon
'}$ are independent. This is an easy consequence of the cascading
equation and Lemma \ref{divmeas}. It follows that, for $h>0$,
\[
G_Q(t+h,T)=G_Q(t,T)+G_Q(h,T-t),
\]
or
\[
\frac{G_Q(t+h,T)-G_Q(t,T)}{h}=\frac{G_Q(h,T-t)}{h}.
\]
Because of Assumption~\ref{good}(3), at $t_0=\ln\frac{1}{\epsilon_0}$
and for all $T\in\mathbb{R}$, the left-hand side converges as $h\to0$,
and so does the right-hand side. It follows that $G_Q$ is
differentiable w.r.t. $t$ at $t=0$, and then at every $t\geq0$.
Furthermore, $\partial_tG(t,T)$ is continuous w.r.t. $(t,T)$ because of
Assumption~\ref{good}(3) again. We deduce that
\begin{eqnarray*}
G_Q(t,T)-G_Q(s,T)&=&\int_{s}^{t}H
\bigl(e^{-T+r}t_1,\ldots,e^{-T+r}t_p
\bigr) \,dr
\\
&=&\int_{e^{s-T}}^{e^{t-T}}H(yt_1,\ldots,yt_p)\frac{dy}{y},
\end{eqnarray*}
where
\[
H(t_1,\ldots,t_p,Q)=\lim_{\epsilon\to1}
\frac{1}{-\ln
\epsilon}\eta^\epsilon(t_1,\ldots,t_p,Q).
\]
Furthermore, $H$ is a continuous function of $(t_1,\ldots,t_p)$. By
taking $T,s=0$ and by noticing that $G_Q(s,T)=0$, we deduce
%e48 ###
%e49 #&#
\begin{equation}
\eta^{\epsilon}(t_1,\ldots, t_p,Q)=\int
_{1}^{1/\epsilon} H(yt_1,\ldots,yt_p)
\frac{dy}{y}.
\end{equation}
Now we want to prove that $H$ stands for the finite-dimensional
distributions of an ID process. For that purpose, observe that
(\ref{fatal}) and (\ref{norm}); that is $\psi_\epsilon(1)=0$, implies
that, for each $t\in\mathbb{R}$, the ID random variable with L\'evy
exponent
\[
q\in\mathbb{R}\mapsto\frac{1}{-\ln\epsilon}\psi^\epsilon(q)
\]
is tight. Indeed, both relations imply that its characteristic triple
$(b_\epsilon,\sigma_\epsilon,\nu_\epsilon)$ satisfies
\[
\sup_\epsilon\biggl(\sigma_\epsilon^2+\int
_\mathbb{R}\min\bigl(1,z^2\bigr)
\nu_\epsilon(dz) \biggr)<+\infty
\]
and
\[
b_\epsilon+
\sigma_\epsilon^2/2+\int_\mathbb{R}
\bigl(e^z-1-z\mathbf{1}_{|z|\leq1}\bigr)\nu_\epsilon(dz)=0.
\]
Hence, for any $(t_1,\ldots,t_p)\in\mathbb{R}^p$, the family of ID
random variables with L\'evy exponents
\[
Q\in\mathbb{R}^p\mapsto\frac{1}{-\ln
\epsilon}\eta^\epsilon(t_1,\ldots,t_p,Q)
\]
is tight since its one-dimensional marginals have $\frac{1}{-\ln
\epsilon}\psi^\epsilon(q)$ as L\'evy exponents. Thus, for any
$(t_1,\ldots,t_p)\in\mathbb{R}^p$, $Q\mapsto H(t_1,\ldots,t_p)$ is
necessarily the L\'evy exponent of some nontrivial
$\mathbb{R}^p$-valued ID random variable. Furthermore, $H$ is
necessarily associated with a consistent family of random variables
since the family $ (Q\mapsto\frac {1}{-\ln
\epsilon}\eta^\epsilon(t_1,\ldots,t_p,Q) )_{p\geq
1,(t_1,\ldots,t_p)\in\mathbb{R}^p}$ is. Hence, there exists an ID
stationary random process $X$ on $\mathbb{R}$ such that $\forall
p\geq1$ and $\forall (t_1,\ldots,t_p)\in\mathbb{R}^p$
\[
\mathbb E \bigl[e^{i q_1 X_{t_1}+\cdots+iq_pX_{t_p}}\bigr]=e^{H(t_1,\ldots,t_p,Q) }.
\]
The Laplace transform $\psi$ of $X_0$ (or equivalently of $X_t$ for any
$t\in\mathbb{R}$) necessarily satisfies
\[
\psi(1)=0 \quad\mbox{and}\quad\psi(1+\delta)\leq\delta.
\]

It remains to prove that $X$ is stochastically continuous. Notice that
the mapping $t\in\mathbb{R}\mapsto H(0,t,Q)$ is continuous. In
particular, we
can choose $Q=(q,-q)$ for some $q\in\mathbb{R}$. We deduce $\lim_{t\to
0} \mathbb E [e^{iq(X_t-X_0)}]=1$ for all $q\in\mathbb{R}$. In
particular, $X_t-X_0$
converges in law toward $0$ as $t\to0$. Therefore $X_t-X_0$ converges
in probability toward $0$ as $t\to0$.

%%%%%%%%%%%%%%%%%%%%%%%%%%%%%%%%%%%%%%%%%%%%%%%%%%%%%%%%%%%%%%%%%%%%%%%%%%%%%%%%%%%%%
%s10 ###
%s10 #&#
\section{\texorpdfstring{Proof of Theorem \lowercase{\protect\ref{cnsq}}}{Proof of Theorem 6}}
%%%%%%%%%%%%%%%%%%%%%%%%%%%%%%%%%%%%%%%%%%%%%%%%%%%%%%%%%%%%%%%%%%%%%%%%%%%%%%%%%%%%%

We explain the proof in dimension $d=1$. The generalization to higher
dimensions is straightforward. Let us consider $\delta>0$ such that $M$
admits a moment of order $1+\delta$. For any $\epsilon\in]0,1[$ and
$n\in\mathbb{N}^*$ with finite Lebesgue measure, we have from Jensen's
inequality,
\[
\mathbb E \biggl[M\biggl(\biggl[0,\frac{1}{n}\biggr]\biggr)^{1+\delta}
\biggr]\geq\mathbb E \biggl[\widetilde{M}{}^\epsilon\biggl(\biggl[0,
\frac{1}{n}\biggr]\biggr)^{1+\delta} \biggr].
\]
We deduce
\[
\mathbb E \bigl[M\bigl([0,1]\bigr)^{1+\delta} \bigr]\geq\sum
_{k=1}^n \mathbb E \biggl[M\biggl(\biggl[
\frac{k-1}{n},\frac{k}{n}\biggr]\biggr)^{1+\delta} \biggr] \geq n
\mathbb E \biggl[\widetilde{M}{}^\epsilon\biggl(\biggl[0,\frac{1}{n}
\biggr]\biggr)^{1+\delta} \biggr].
\]
Let us define for $\epsilon\in]0,1[$ and $q\leq1$
\[
F_\epsilon(q,r)=\ln\mathbb E \bigl[e^{qX^\epsilon_0+qX^\epsilon
_r-2q\psi(1)\ln(1/\epsilon)}\bigr]
\]
and observe that
\[
F_\epsilon(q,r)=\int_1^{1/\epsilon}\bigl[
\psi_{0,rg(y)}(q,q)-2q\psi(1)\bigr]\frac{dy}{y}.
\]
Let us consider $h>1$ such that $(1+\delta)h=2$. By using the concavity
of the function $x \mapsto x^{1/h}$ and Jensen's inequality, we get
\begin{eqnarray*}
&& \mathbb E \biggl[ \widetilde{M}{}^\epsilon\biggl(\biggl[0,
\frac{1}{n}\biggr] \biggr)^{1+\delta} \biggr]
\\
&&\qquad  = \mathbb E \biggl[\widetilde{M}{}^\epsilon\biggl(\biggl[0,\frac
{1}{n}\biggr]
\biggr)^{2\times(1/h)} \biggr]
\\
&&\qquad = \mathbb E \biggl[ \biggl( \int_{0}^{1/n} \int
_{0}^{1/n}e^{X^\epsilon_r-\psi(1)\ln(1/\epsilon)+X^\epsilon
_u-\psi(1)\ln(1/\epsilon)} \,dr\,du
\biggr)^{1/h} \biggr]
\\
&&\qquad  \geq n^{-1-\delta} \mathbb E \biggl[ \int_{0}^{1/n}
\int_{0}^{1/n}e^{(1/h)X^\epsilon_r+(1/h)X^\epsilon
_u-(1+\delta)\psi(1)\ln(1/\epsilon)}
n^2\,dr\,du \biggr]
\\
&&\qquad \geq n^{-1-\delta}e^{\inf_{|r|\leq 1/n} F_\epsilon(1/h,r)}.
\end{eqnarray*}
Gathering the above inequalities yields
%e49 ###
%e50 #&#
\begin{equation}
\label{relfundter} \mathbb E \bigl[M\bigl([0,1]\bigr)^{1+\delta} \bigr
]\geq
n^{-\delta}e^{\inf
_{|r|\leq 1/n} F_\epsilon(1/h,r)}.
\end{equation}

Now fix $\alpha>0$. Because the mapping $u\mapsto
\psi_{0,u}(\frac{1}{h},\frac{1}{h})$ is continuous, there exists
$\eta>0$ such that
$|\psi_{0,u}(\frac{1}{h},\frac{1}{h})-\psi_{0,0}(\frac{1}{h},\frac
{1}{h})|\leq
\alpha$ for $|u|\leq\eta$. We choose $\epsilon=\frac{1}{n \eta}$ and
we obtain for $|r|\leq\frac{1}{n}$,
\begin{eqnarray*}
\biggl|F_\epsilon\biggl(\frac{1}{h},r\biggr)-F_\epsilon\biggl(
\frac{1}{h},0\biggr)\biggr|&\leq& \int_1^{n \eta}\biggl|
\psi_{0,rg(y)}\biggl(\frac{1}{h},\frac{1}{h}\biggr)-\psi
_{0,0}\biggl(\frac{1}{h},\frac{1}{h}\biggr)\biggr|
\frac{dy}{y}
\\
&\leq& \alpha\ln(n\eta).
\end{eqnarray*}
We deduce
\[
\inf_{|r|\leq 1/n} F_\epsilon\biggl(\frac{1}{h},r
\biggr)\geq F_\epsilon\biggl(\frac{1}{h},0\biggr)-\alpha\ln(n\eta)=
\ln(n\eta) \bigl(\psi(1+\delta)-(1+\delta)\psi(1)-\alpha\bigr).
\]
By plugging this relation into (\ref{relfundter}), we get
\[
\mathbb E \bigl[M\bigl([0,1]\bigr)^{1+\delta} \bigr]\geq
n^{-\delta}e^{\ln
(n\eta) (\psi(1+\delta)-(1+\delta)\psi(1)-\alpha)}.
\]
Since this relation must be valid for all $n$ large enough, we
necessarily have
\[
\psi(1+\delta)-(1+\delta)\psi(1)-\alpha\leq\delta.
\]
Since $\alpha>0$ is arbitrary, we deduce
\[
\psi(1+\delta)-(1+\delta)\psi(1) \leq\delta.
\]
In particular, by convexity arguments [as in establishing
(\ref{fatal})] we have
\[
\psi'(1)-\psi(1)< 1.
\]
%
%
%For all $\epsilon', \epsilon$, we get the following:
%e^{\omega_{\epsilon' \epsilon}(x)} M_{\epsilon' \epsilon}(dx)= e^{
%By taking $E[.| \mathcal{F}_{\epsilon\epsilon'} ]$, we get:
%(\omega_{\epsilon' \epsilon}(x))_{x \in\R} = (\omega_{ \epsilon}(x)+
%and therefore for all $x$, $t \rightarrow\omega_{e^{-t}}(x)$ is a
%continuous Levy process and thus a Brownian motion with drift $\nu$
%(starting from $0$). In particular, the variable $\omega_{
%(\omega_{\epsilon\epsilon'}(x))_{x \in\R}\stackrel{law}{=} (\omega_{
%where $\omega_{\epsilon}$ and $\tilde{\omega}_{\epsilon'}$ are
%independent. In the sequel, we suppose that the process $(\omega_{
%We consider $\epsilon_n=e^{-\frac{1}{n}}$. Of course $
%(\omega_{e^{-1}}(x))_{x \in\R}\stackrel{law}{=} (\sum_{k=0}^{n-1}
%where the $\omega_{\epsilon_n}^{(k)}$ are independent processes of law
%$\omega_{\epsilon_n}$. Fix $x \in\R$. We therefore have for all $
%Taking covariances, we get the following identity $\E[
%
%Now it is easy to conclude by Fourier for example:
% \E[ e^{ i\lambda\omega_{e^{-1}}(0)+i \mu\omega_{e^{-1}}(x) }]& =
%& \approx\Pi_{k=0}^{n-1} (1-\frac{1}{2} \E[ (\lambda\omega_{
%& = \Pi_{k=0}^{n-1} e^{ \ln(1-\frac{\lambda^2}{2n}-\frac{\mu^2}{2n}-
%& \approx\Pi_{k=0}^{n-1} e^{ -\frac{\lambda^2}{2n}-\frac{\mu^2}{2n}-
%& = e^{-\frac{\lambda^2}{2}-\frac{\mu^2}{2}-\lambda\mu\E[
%
%This shows that the couple $(\omega_{e^{-1}}(0),\omega_{e^{-1}}(x)) $
%is a Gaussian vector. One proceeds similarly to show that, for all $
\end{appendix}

\section*{Acknowledgements} The authors would like to thank
Hubert Lacoin for useful discussions.

%suskaldyti doi

% imsref loaded by linak, 2013-12-02 16:45:00

\printaddresses

\end{document}